\documentclass[leqno,11pt]{article}
\usepackage{microtype}

\usepackage{amsmath}
\usepackage{amsfonts}
\usepackage{amssymb}
\usepackage{fullpage}
\usepackage{hyperref}

\usepackage{amsmath,amsthm,amssymb,mathrsfs} 
\usepackage[T1]{fontenc}
\usepackage{ae,aecompl}
\usepackage{times}
\usepackage[english]{babel}
\usepackage{bookmark}

\newtheorem{theorem}{Theorem}

\newtheorem{corollary}{Corollary}

\newtheorem{definition}{Definition}
\newtheorem{example}{Example}

\newtheorem{lemma}{Lemma}

\newtheorem{proposition}{Proposition}
\newtheorem{remark}{Remark}

\numberwithin{equation}{section}

\author{Goncalo Oliveira \\  Duke University and Imperial College London}

\title{Monopoles on AC $3$-manifolds}

\begin{document}
\maketitle

%===============================================================================
\begin{abstract}
We construct monopoles on asymptotically conical (AC) $3$-manifold $X$ with vanishing second Betti number $b^2(X)$. For sufficiently large mass, our construction covers an open set in the moduli space of monopoles.
\\
%These monopoles are then parametrized by $U \times \mathbb{T}^k \times H^1(X, \mathbb{S}^1)$, which one may interpret as follows: $U \subset X^k$ is an open set parametrizing the location of the $k$ monopoles, $\mathbb{T}^k$ are their local phases and $H^1(X, \mathbb{S}^1)$ a possible twist with a flat connection.\\
We also give a more general construction of Dirac monopoles in any AC manifold, which may be useful for generalizing our result to the case when $b^2(X) \neq 0$.
\end{abstract}
%===============================================================================

%===============================================================================

Let $(X^3,g)$ be a $3$ dimensional Riemannian manifold, $P \rightarrow X$ an $SU(2)$-bundle and $\mathfrak{g}_P$ the associated adjoint bundle, equipped with an $Ad$-invariant metric. A pair $(A, \Phi)$ where $A$ is a connection on $P$ and $\Phi \in \Omega^0(X,\mathfrak{g}_P)$ called the Higgs field is said to be a monopole if they satisfy
\begin{equation}\label{eq:monopole}
\ast F_A - \nabla_A \Phi=0.
\end{equation}
In this paper we study asymptotically conical (AC) $3$ manifolds $(X,g)$. These are noncompact and it is helpful to think of them as the interior of a manifold with boundary $\Sigma=\overline{X}$, then the metric $g$ is modeled on a cone metric on $\mathbb{R} \times \Sigma$, see definition \ref{def:AC} for the details. Let $\Sigma = \partial \overline{X}$ and denote by $\Sigma_i$, for $i \in \lbrace 1, ... , l \rbrace $ its connected components. We shall refer to the $i$-th end of $X$ has the one modeled on $(1,+ \infty) \times \Sigma_i$, and let $\rho$ denote a radial coordinate which at each end is identified with the radial distance along the cone. Suppose that at each end, connection $A$ converges to some connection $A_{\infty}$ pulled back from a bundle over the $\Sigma$. In this situation, a monopole $(A,\Phi)$ is said to have finite mass if at each end, $i$ say, there is $m_i \in \mathbb{R}^+$ with $\vert \Phi \vert$ converging uniformly in $\Sigma_i$ to $m_i$. Under such conditions one can prove \cite{Jaffe1980}, \cite{Oliveira2014}
\begin{eqnarray}\label{bc1}
\vert \Phi \vert = m_i - \frac{k^{O}_i}{2\rho} + O(\rho^{-2}) ,
\end{eqnarray}
for some $m_i \in \mathbb{R}$ and $k^O_i \in \mathbb{N}$ (in fact more is known regarding the asymptotic behavior of finite mass monopoles and we refer the reader to previous references). We shall call the tuples $\vec{k}=(k^O_1,...,k^O_l)$ and $\vec{m}=(m_1, ...,m_{l}) \in \mathbb{R}^{l}$ the charge and the mass of the monopole $(A,\Phi)$ respectively. The integer $k = k^{O}_1 + ... + k^O_{l}$ will be called the total charge.

\begin{remark}
In fact one can prove, \cite{Oliveira2014} that under the assumption of finite mass the Higgs field converges to a parallel section of $\mathfrak{g}_P$ at each end. Moreover, if it does not vanish, i.e. $m_i \neq 0$, then $A$ converges to reducible Yang-Mills connections at the ends. Having this in mind we remark that each of the $k^O_i$ is the degree of the complex line bundles over the $\Sigma_i$'s associated with each of such reductions.
\end{remark}

\begin{definition}
Denote the space of smooth pairs $(A, \Phi)$ with charge $\vec{k} \in \mathbb{Z}$ and mass $\vec{m} \in \mathbb{Z}^{l}$ by $C_{\vec{k},\vec{m}}(X,g)$, and by $M_{\vec{k},\vec{m}}$ those $(A,\Phi) \in C_{\vec{k},\vec{m}}$ solving the monopole equation \ref{eq:monopole}.\\
Fix $(A,\Phi) \in M_{\vec{k},\vec{m}}$ and let $\mathcal{G}$ be the space of gauge transformations $g \in Aut(P)$ with $g^{-1} \nabla_{A} g$ in the Sobolev space $H^2_{2,1/2}$ as in definition \ref{def:FunctionSpaces}. Equation \ref{eq:monopole} is invariant under the $\mathcal{G}$-action and we define the moduli space of total charge $k$ and mass $m$ monopoles.
$$\mathcal{M}_{k,m}(X,g)= M_{k,m}(X,g)/ \mathcal{G}.$$
\end{definition}

The main result of this paper is

\begin{theorem}\label{MainTheorem}
Let $k \in \mathbb{Z}$ and $(X,g)$ be an asymptotically conical $3$-manifold with $b^2(X)=0$ (notice this implies $l=1$). Then, there is $\mu \in \mathbb{R}$, such that if $m \geq \mu$ and $X^k(m) \subset X^k$ denotes the open set defined by
$$X^k(m)= \Big\lbrace (p_1,...,p_k) \in X^k \ \Big\vert \ dist(p_i,p_j) > 4 m^{-\frac{1}{2}} \ , \ \text{for $i \neq j$} \Big\rbrace,$$
while $\check{\mathbb{T}}^{k-1}= \lbrace (e^{i\theta_1}, ... , e^{i\theta_k}) \in \mathbb{T}^k \ \vert \ e^{i(\theta_1+ ... + \theta_k)} =1 \rbrace$ there is a map
\begin{equation}\label{eq:Map}
\tilde{h}: X^k(m)  \times H^1(X, \mathbb{S}^1) \times \check{\mathbb{T}}^{k-1} \rightarrow M_{k,m},
\end{equation}
i.e. its image consists of monopoles.
\end{theorem}

Moreover, this construction is then showed to descend well to the moduli space of monopoles

\begin{theorem}\label{MainTheorem2}
The map $h$ in equation \ref{eq:Map} descends to a local diffeomorphism
$$h: X^k(m) \times   H^1(X, \mathbb{S}^1) \times \check{\mathbb{T}}^{k-1} \rightarrow \mathcal{M}_{k,m}$$
\end{theorem}

\begin{remark}
\begin{enumerate}
\item We may interpret the monopoles in the image of the map $h$ as being formed by gluing $k$ well separated monopoles. Then one may think of the parameters in $X^k(m) \times   H^1(X, \mathbb{S}^1) \times \check{\mathbb{T}}^{k-1}$ used in their construction as follows: The points $(p_1,...,p_k) \in X^k(m)$ denote the location of the $k$ monopoles, $\check{\mathbb{T}}^{k-1}$ denotes local phases assigned to each of these and $H^1(X, \mathbb{S}^1)$ a possible twist with a flat connection, i.e. a usual vector potential whose strength (magnetic field) vanishes inside the superconductor filling in $X$.
\item In \cite{Kottke13}, Kottke computes the virtual dimension of the moduli space of monopoles with fixed mass on an AC $3$-manifold and obtains the formula $dim (\mathcal{M}_{m,k})= 4k + \frac{1}{2} b^1(\Sigma)-b^0(\Sigma)$. The long exact sequence on cohomology $H_c(X) \rightarrow H(X) \rightarrow H(\Sigma)$ and the duality $H^*_c \cong (H^{3-*})^*$, can be used to rewrite Kottke's formula as
$$\dim(\mathcal{M}_{m,k})= 4k + b^1(X)- b^2(X) -1,$$
and the construction described here gives a good geometric interpretation of all these parameters in an open set in $\mathcal{M}_{m,k}$. Even though it will be evident that the elements of $b^2(X)$ obstruct the gluing construction, the author believes the analysis here can be changed to take these into account and extend theorems \ref{MainTheorem} and \ref{MainTheorem2} to the case of $b^2(X) \neq 0$.
\end{enumerate}
\end{remark}

We shall now give a short outline of the paper and of the proof of the results above. The first result of the paper is a construction of certain Dirac monopoles. This requires definition \ref{bMap} which constructs $X^k(m) \subset X^k$ as the intersection of a $3k-b^2(X)$ dimensional submanifold with a big open set in $X^k$. Then, for each element in $X^k(m) \times H^1( X , \mathbb{S}^1)$ theorem \ref{DiracMonopole} constructs a Dirac monopole. Later in section section \ref{ApproximateSolution}, proposition \ref{prop:AproximateSolution}, these Dirac Monopoles are smoothed out around the singular points $p_i \in X$. To smooth each singular point requires an element of $\mathbb{S}^1$ and this construction gives a map
\begin{eqnarray}
H:  X^k(m) \times H^1( X , \mathbb{S}^1) \times \check{\mathbb{T}}^{k-1} \rightarrow C_{k,m} (X,g),
\end{eqnarray}
whose image consists of approximate solutions to the Bogomolny equations. Still in in section \ref{ApproximateSolution} we estimate the error term $e_0=e(A_0,\Phi_0)$ of a configuration in the image of the map $H$. Section \ref{Proof} constructs a general setup to solve the monopole equation by deforming an approximate solution $(A_0,\Phi_0)$ in the image of the map $H$. More precisely, section \ref{FunctionSpaces} constructs Function Spaces depending on the approximate solution $(A_0,\Phi_0)$. These are specially adapted to uniformly invert the linearized equation, make the the error term $e_0$ small, and handle the nonlinearities. Then, in section \ref{LinearEquation} the linearised equation is shown to admit a right inverse with uniformly bounded norm on the previously defined function spaces. Finally, section \ref{NonlinearEquation} solves the monopole equation, see proposition \ref{prop:MonopoleSolution}, by using a special version of the contraction mapping principle.\\
To the author's knowledge the first reference in the literature investigating monopoles on a large class of $3$-manifolds is Braam's paper \cite{Braam1989}, which explores monopoles in asymptotically hyperbolic manifolds. Monopoles on asymptotically Euclidean $3$-manifolds have been investigated by Floer \cite{Floer1995} and Ernst \cite{Ernst1995}. The main motivation for all these is the possibility of using the moduli space of monopoles on noncompact $3$-manifolds to attack problems in $3$-dimensional topology. This tries to imitate the way instantons have shed light on $4$ dimensional topology and this possibility have remained unexplored. In this direction here we have extended Floer's and Ernst to the more general class of asymptotically conical manifolds where we are able to describe an open set of the moduli space.

\subsection{Acknowledgements}

The work presented here was done while I was a PhD student at Imperial College London. I would like to thank Simon Donaldson for guidance and encouragement. I would also like to thank Chris Kottke, Lorenzo Foscolo and Michael Singer for discussions. I am very grateful for having had my PhD supported by the FCT doctoral grant with reference SFRH / BD / 68756 / 2010.

\section{Preliminary Remarks}

\subsection{Linearized Operator}\label{sec:LinearizedOp3}

Let $(\nabla_{A}, \Phi)$ be a connection and Higgs field, not necessarily satisfying the Bogomolny equations. For such a configuration the quantity $e_0 = \ast F_{A} - \nabla_{A} \Phi$ may be nonzero. The linearized Bogomolny equation fits into a sequence
\begin{equation}\label{Complex}
\Omega^0(\mathfrak{su}(P)) \overset{d_1}{\rightarrow} \Omega^1(\mathfrak{su}(P)) \oplus \Omega^0(\mathfrak{su}(P)) \overset{d_2}{\rightarrow} \Omega^1(\mathfrak{su}(P)),
\end{equation}
with $d_1\xi = (- \nabla_A \xi , - [\Phi, \xi]) $ and
\begin{eqnarray}\label{eq:LinearizedMonopoleeq}
d_2(a, \phi) = \ast d_A a - \nabla_A \phi - [a, \Phi].
\end{eqnarray}
Their formal adjoints are given by $d_1^* (a, \phi) = - \nabla_A^* a + [\Phi,\phi]$ and $d_2^* a = (\ast d_A a + [a, \Phi], - \nabla^*_A a )$. If $(A, \Phi)$ is a monopole then the sequence in \ref{Complex} is actually an elliptic complex and so the operator $D= d_2 \oplus d_1^*$ acting on sections of $(\Lambda^1 \oplus \Lambda^0)(\mathfrak{su}(P))$ is elliptic. Its formal adjoint is $D^* = d_2^* \oplus d_1$ and these can be written as
$$D= \begin{pmatrix}
\ast d_A & - \nabla_A \\
-d_A^* & 0
\end{pmatrix} + [\Phi ,.] \ \ , \ \ D^* = D- 2[\Phi, .].$$

\begin{lemma}(Standard Weitzenb\"ock)
Let $\nabla_A$ be a connection and $u \in \Omega^1(\mathfrak{su}(2)) \oplus \Omega^0(\mathfrak{su}(2))$, then
\begin{equation}
\Delta_A u =\nabla_A^* \nabla_A u + F^W (u) + Ric^W (u).
\end{equation}
Where $F^W(a, \phi) = \left( \ast [\ast F_A \wedge a], 0 \right)$ and $Ric^W (a, \phi) = ( Ric(a), 0)$.
\end{lemma}

\begin{lemma}\label{lem:MonopoleW3}(Monopole Weitzenb\"ock)
Let $(\nabla_A, \Phi)$ be a connection and an Higgs Field. Let $u \in \Omega^1(\mathfrak{su}(2)) \oplus \Omega^0(\mathfrak{su}(2))$, then
\begin{eqnarray}\label{MonopoleW}
DD^* u & = & \nabla_A^* \nabla_A u - [[u , \Phi], \Phi]  + Ric^W (u) + \epsilon_0^W(u) \\
D^*D u & = & DD^* u + 2(\nabla_A \Phi)^W(u).
\end{eqnarray}
Where $b^W (a, \phi) = \left( \ast [a \wedge b] - [b, \phi] , [ \langle b ,  a\rangle ] \right)$ and $b$ is either $\epsilon_0 = \ast F_A - \nabla_A \Phi$, $Ric$ or $(\epsilon_0 + 2d_A \Phi)$.
\end{lemma}

If $(A,\Phi)=  (A_0 + a,\Phi_0 + \phi )$ for suitable $u=(a, \phi) \in \Omega^1(\mathfrak{su}(P)) \oplus \Omega^0(\mathfrak{su}(P))$ is a monopole, then 
\begin{equation}\label{firstorderPDE}
 \epsilon_0 + D(u) + Q(u,u)=0,
\end{equation}
where the operator $D$ is as above and $Q(u,u) = \begin{pmatrix} \ast [ a \wedge a ] - [a , \phi ] \\ 0 \end{pmatrix}$, i.e. the nonlinear terms appearing here in the gauge-fixed monopole equation are zero order and quadratic.\\

We turn now to one other very important property of monopoles. This is the scale invariance of the Bogomolny equation, inherited from the conformal invariance of the ASD equations in $4$ dimensions. The precise result is

\begin{proposition}\label{scaling}
Let $(\nabla_A, \Phi)$ be a monopole on $(M^3,g)$, where $M^3$ is a Riemannian $3$ manifold. Then $(\nabla_A , \delta^{-1} \Phi)$ is a monopole for $(M^3 , \tilde{g} = \delta^2 g)$.
\end{proposition}
\begin{proof}
In general, if $\omega$ is a $k$ form and $\tilde{\ast}$ the Hodge operator for the metric $\tilde{g}$, then $\tilde{\ast} \omega = \delta^{n-2k} \ast \omega$ ($n=3$). This implies that $\tilde{\ast} F_A = \delta^{-1} \ast F_A = \delta^{-1} \nabla_A \Phi$, and the result follows.
\end{proof}

The Yang-Mills-Higgs (YMH) Energy is defined on a precompact set $U \subset X$ as
\begin{equation}
E_U = \frac{1}{2}\int_U \vert \nabla_A \Phi \vert^2 + \vert F_A \vert^2.
\end{equation}
The Euler Lagrange equations are $d_A^{\ast} F_A = \left[ d_A \Phi, \Phi \right]$, $\Delta_{d_A} \Phi =0$. It is well known that monopoles not only solve these, but also minimize the YMH energy on $X$. The energy over a precompact set $U$ with smooth boundary is given by the flux $\int_{\partial \overline{U}} \langle \Phi , F_A \rangle$. On an AC manifold finite mass monopoles have finite energy given by $E_U = 4 \pi \sum_{i=1}^k m_i k^O_i$, see \cite{Oliveira2014} for a proof.

\subsection{Analytical Preliminaries on AC Manifolds}

\begin{definition}\label{def:AC}
A $3$ dimensional Riemannian manifold $(X^n,g)$ is called asymptotically conical (AC) with rate $\nu < 0$, if there is a compact set $K \subset X$, a Riemann surface $(\Sigma^{2}, g_{\Sigma})$ and a diffeomorphism
$$\varphi : (1 , \infty) \times \Sigma \rightarrow X \backslash K,$$
such that the metric $g_C = dr^2 + r^2 g_{\Sigma}$ on $(1 , \infty) \times \Sigma$ satisfies $\vert \nabla^{j} (\varphi^* g - g_C) \vert_{C} = O (r^{\nu-j})$, for all $j \in \mathbb{N}_0$. Here $\nabla$ is the Levi Civita connection of $g_C$. A radius function will be any positive function $\rho : X \rightarrow \mathbb{R}_{+}$, such that  in $X \backslash K$, $\rho = r \circ \varphi^{-1}$.
\end{definition}

In particular $\Sigma$ may be not be connected. In that case we can write $\Sigma$ as the disjoint union of its connected components, $\Sigma = \cup_{i=1}^{l} \Sigma_i$. Then we shall refer to $\varphi ((1, + \infty) \times \Sigma_i)$ as the $i$-th end of $X$.

The rest of this section contains a brief discussion of the Lockhart-McOwen conically weighted spaces, see chapter $4$ in \cite{Mar02} for more details. For $p>0$, $n \in \mathbb{N}$ and $\beta \in \mathbb{R}$ define the norms
$$\Vert f \Vert_{L^p_{n,\beta}}^p = \sum_{i \leq n} \int_X \vert \rho^{-\beta-i} \nabla^i f \vert^p \rho^{-3} dvol.$$

\begin{definition}
Let $L^p_{n, \beta}$ denote the Banach space completion of the smooth compactly supported sections in the norm $\Vert \cdot \Vert_{L^p_{n,\beta}}$ defined above.
\end{definition}

Later in section \ref{FunctionSpaces} we shall define Hilbert spaces $H_{n,\beta}$ such that $d_2: H_{1,\beta} \rightarrow H_{0,\beta-1}$ is surjective and admits a bounded right inverse for some $\beta \in \mathbb{R}$. To prove that this is the case we shall need to analyze the operator
\begin{equation}\label{eq:SignatureOp}
d+d^* : L^2_{1,\beta} \rightarrow L^2_{0, \beta -1}.
\end{equation}

\begin{lemma}\label{DiracBoundII}
Let $(X^3,g)$ be AC and denote by $g_{\Sigma}$ the metric on the link $\Sigma$ at the conical end of $X$, then the following statements hold.
\begin{enumerate}
\item There is a discrete set
$$D(d+d^*)= \lbrace \beta \in \mathbb{R} \ \vert \ ( \beta +1 ) (3-1 + \beta) \in Spec(\Sigma)_0 \rbrace,$$
such that if $\beta \in \mathbb{R} \backslash  D(d+d^*)$ the operator \ref{eq:SignatureOp} on $1$-forms is Fredholm. Here $Spec(\Sigma)_0$ denotes the spectrum of the Laplacian $\Delta_{g_{\Sigma}}$ on functions.
\item For $\beta \leq - \frac{1}{2}$, $ker(d+d^*)_{\beta} = ker(\Delta)_{\beta}$ and for $1$-forms, there is an isomorphism $ker(\Delta)_{-\frac{3}{2}} \cong H^1_c(X, \mathbb{R}) \cong H^2(X, \mathbb{R})^*$.
\end{enumerate}
\end{lemma}
\begin{proof}
For the first item see from \cite{Mar02}, chapter $6$. The weights appearing in $D(d+d^*)$ correspond to the rates of the homogeneous closed and coclosed $1$-forms on the metric cone $(\mathbb{R}^+ \times \Sigma, g_C=dr^2 + r^2 g_{\Sigma})$. The second item follows from the fact that $L^2_{0, - 3/2}=L^2$ and the statement that there is an isomorphism between the space of $L^2$ harmonic $1$-forms and the compactly supported cohomology, see \cite{Lockhart1987}.
\end{proof}

\begin{corollary}\label{rmk:LM}
Let $(X^3,g)$ be AC with $b^2(X)=0$, then there is a constant $c>0$ such that for $\alpha < -1$ and all $u \in L^2_{1,\alpha}$
$$\Vert (d+d^*) u \Vert_{L^2_{0, \alpha -1}} \geq c \Vert  u \Vert_{L^2_{1, \alpha}}$$
\end{corollary}
\begin{proof}
It follows from the fact that $d+d^*$ is Fredholm for those $\alpha \not\in D(d+d^*)$, that $L^2_{0, \alpha -1}=(d+d^*)(L^2_{1,\alpha}) \oplus W_{\alpha -1}$, with $W_{\alpha-1} \cong ker(d+d^*)_{-2-\alpha}$. Since $d+d^*$ has closed image for $\alpha \notin D(d+d^*)$, there is $c>0$ such that
$$\Vert (d+d^*) u \Vert_{L^2_{0, \alpha -1}} \geq c \Vert  u \Vert_{L^2_{1, \alpha}},$$
for all $u \in ker(d+d^*)_{\alpha}^{\perp}$. Moreover, for $\alpha \leq - \frac{3}{2}+1$, we can integrate by parts and so $ker(\Delta)_{\alpha}= ker(d + d^*)_{\alpha}$. This together with the second item in lemma \ref{DiracBoundII} gives $ker(d+d^*)_{\alpha} \subset ker(\Delta)_{-\frac{3}{2}} \cong H^2(X, \mathbb{R})^*$, for all $\alpha \leq -\frac{3}{2}$. Then, the assumption that $b^2(X)=0$ finally gives $ker(d+d^*)_{\alpha} =0$
for all $\alpha \leq -\frac{3}{2}$. However, as the Laplacian $\Delta_{\Sigma}$ has no negative eigenvalues, the first item in lemma \ref{DiracBoundII} gives that there are no critical rates $\alpha$ in the interval $(-2,-1)$. Hence, one can increase $\alpha$ up until (but excluding) $-1$.
\end{proof}

%This result is sufficient for our application, however there is a more general result which may be useful in extending the study of monopoles to the case, when $b^2(X) \neq 0$, we state this below:
%
%\begin{proposition}\label{prop:Marshal}(S. Marshal, \cite{Mar02})
%Let $(X,g)$ be asymptotically conical and consider $\beta \in \mathbb{R} \backslash D(d+d^*)$ the set in lemma \ref{DiracBoundII}. Then
%\begin{equation*}
%\ker(d+d^*)_{\beta}\cong \left\{
%\begin{array}{rl}
%d \ker(\Delta_{\beta+2}) \oplus H^1(X, \mathbb{R}), & \text{if } -1< \beta ,\\
%H^1_{cs}(X, \mathbb{R}) , & \text{if } -2 < \beta < -1,\\
%im (H^1_{cs}(X) \rightarrow H^1(X)) & \text{if }\beta < -2,
%\end{array} \right.
%\end{equation*}
%where $\Delta$ above denotes the Laplacian on functions. In particular, let $\lambda \in \mathbb{R}$ and $\chi(\lambda)= \lbrace dim ( \ker(\Delta_{\Sigma} - \mu) ) \ \vert \ \mu \leq \lambda(\lambda+3)\rbrace$, then 
%\begin{equation*}
%dim (\ker(d+d^*)_{\beta}) = \left\{
%\begin{array}{rl}
%l -1 + b^1(X) + \chi(\beta+1), & \text{if } -1< \beta ,\\
%b^1_{cs}(X, \mathbb{R})=b^2(X) , & \text{if } -2 < \beta < -1,\\
%b^1_{cs}(X) -l +1 & \text{if }\beta < -2,
%\end{array} \right.
%\end{equation*}
%\end{proposition}
%\begin{proof}
%See chapter $5$ in \cite{Mar02}.
%\end{proof}

The following two results will also be used later during the construction of monopoles on AC $3$ manifolds and it is convenient to have them stated now.

\begin{lemma}\label{Kato}
Let $\nabla_{A}$ be a metric compatible connection on an Hermitian vector bundle $E$ over an AC manifold $(X^3 , g)$. Then, for all $\alpha \in [1,3]$, there is a constant $c_K(\alpha) > 0$, such that
$$\left( \int_X \vert \rho^{\frac{1}{2}} u \vert^{2 \alpha} \frac{dvol_g}{\rho^3}\right)^{\frac{1}{2\alpha}} \leq c_K(\alpha) \left( \int_X \vert \nabla_A u \vert^2 \right)^{\frac{1}{2}}$$
for all smooth and compactly supported section $u$. In particular for $\alpha =3,1$ one has respectively $\Vert u \Vert_{L_6}^2 \leq c_K \Vert \nabla_A u \Vert^2_{L_2}$ and $\Vert \rho^{-1} u \Vert_{L_2}^2 \leq c_K \Vert \nabla_A u \Vert^2_{L_2}$.
\end{lemma}
\begin{proof}
Kato's inequality $\vert \nabla \vert u \vert \vert \leq \vert \nabla_A u \vert$, holds pointwise for all irreducible Hermitian connections. The proof follows from combining this with corollary $1.3$ in \cite{Hein11}.
\end{proof}

\begin{lemma}\label{LimitLemma}
In the conditions of lemma \ref{Kato}. Let $u$ be a section such that $\nabla_A u \in L^2$, then there is a covariant constant limit $u\vert_{\Sigma} \in \Gamma(\Sigma, E \vert_{\Sigma})$. Moreover, on the cone $C(\Sigma_i)$ over each end there is an inequality
$$\Vert \vert u \vert - u_{\Sigma_i} \Vert_{L^{2\alpha}_{0, -\frac{1}{2}} } \leq \Vert \nabla_A u \Vert_{L^2}.$$
\end{lemma}
\begin{proof}
This lemma is a particular case of propositions $A.0.16$ and $A.0.17$ in the Appendix A to \cite{Oliveira2014}.
\end{proof}

\subsection{Monopoles on $(\mathbb{R}^3, g_E)$}\label{MonopolesR3}

In this short section $\mathcal{M}_{k,m}$ denotes the moduli space of charge $k \in \mathbb{Z}$ and mass $m \in \mathbb{R}^+$ monopoles in Euclidean $\mathbb{R}^3$. In the construction of the approximate solution it will be important to scale these monopoles. Given $\lambda \in \mathbb{R}^+$, there is a bijection $\mathcal{M}_{k,m} \rightarrow \mathcal{M}_{k,\lambda m}$, which can be described as follows. Let $\lambda \in \mathbb{R}^+$ and $(A,\Phi)$ a monopole on $(\mathbb{R}^3, g_E)$, then $(A, \lambda \Phi)$ is a monopole on $(\mathbb{R}^3,\lambda^{-2} g_E)$. Using the scaling map $\exp_{\lambda}(x)= \lambda x$ on $ \mathbb{R}^3$ we define
$$\Phi_{\lambda} = \lambda \exp_{\lambda}^* \Phi \ , \ A_{\lambda} = \exp_{\lambda}^* A \ , \ (g_E)_{\lambda}= \lambda^{-2} \exp_{\lambda}^* g_E.$$
Since the Euclidean metric is invariant under scaling $(g_E)_{\lambda}=g_E$, and $(A_{\lambda},\Phi_{\lambda} )$ is a monopole for the Euclidean metric. However, the monopole $(A_{\lambda},\Phi_{\lambda} )$ no longer has mass $m$ but mass $\lambda m$.

Let $\underline{\mathbb{C}^2}$ denote the trivial rank $2$ complex vector bundle, then an isomorphism $\eta : \underline{\mathbb{C}^2} \vert_{\mathbb{R}^3 \backslash \lbrace 0 \rbrace} \rightarrow H^k \oplus H^{-k}$ is called a framing. We shall fix a framing $\eta$ which identifies the limiting Higgs Field and connection with the pullbacks via $\eta$ of those determined by the unique $SU(2)$ invariant configurations on the Hopf bundle $H$. These framings are unique up to a factor of $\mathbb{S}^1 / \mathbb{Z}_k$, where $\mathbb{S}^1$ denotes the automorphism group of $H$ equipped with its unique $SU(2)$ invariant connection. Moreover, such a framing $\eta$ also gives an isomorphism $\underline{\mathfrak{su}(2)}\vert_{\mathbb{R}^3 \backslash \lbrace 0 \rbrace} \rightarrow \underline{\mathbb{R}} \oplus H^2$. Hence, given a section $a$ of the adjoint bundle, one can write $a= a^{\Vert} \oplus a^{\perp}$ according to this splitting. These components will be respectively called the longitudinal and the transverse one.

\begin{example}\label{1MonopoleExample}
Fixing a framing gives a model for the Hopf bundle $H$, this together with the map that to a charge one monopole assigns the zero of the Higgs field, gives a map $M_{1,1} \cong \mathbb{R}^3 \times \mathbb{S}^1$. Hence, one can consider monopoles with a fixed center and then $\mathring{M}_{1,1} \cong \mathbb{S}^1$.
\end{example}

\begin{definition}\label{def:BPSandR}
There is a unique spherically symmetric monopole with mass $1$, \cite{BPS}, which we shall call the BPS monopole and denote it by $(A^{BPS}, \Phi^{BPS})$. Moreover, one remarks that there is $R>0$ such that $\vert \Phi^{BPS} \vert > 1/2$ outside the ball of radius $R$ in $\mathbb{R}^3$.
\end{definition}

Then, any element of $ \mathring{M}_{1,1}$ is a framing $\eta: \underline{\mathbb{C}^2} \vert_{\mathbb{R}^3 \backslash \lbrace 0 \rbrace} \rightarrow H \oplus H^{-1}$ which at infinity identifies the connection $A^{BPS}$ with the direct sum of the unique $SU(2)$ invariant connection on $H$. Hence, fixing a random framing $\eta_0$ any other $\eta \in  \mathring{M}_{1,1}$ is such that $\eta \circ \eta_0^{-1}$ is multiplication by a constant function with values in $\mathbb{S}^1$. This gives an isomoprhism $ \mathring{M}_{1,1} \cong \mathbb{S}^1$ and from now on we shall think of $ \mathring{M}_{1,1}$ as being $\mathbb{S}^1$. 

The following lemma is an important tool for estimating the error term of the approximate solution

\begin{lemma}\label{lem:BPSestimate}
Let $(A^D , \Phi^D)$ be the Dirac monopole on $\mathbb{R}^3$ with mass $1$. Then, for all $k$ and $(A, \Phi) \in C_{k,1}$, there is $\nu$, such that
$$
\vert \left( \Phi - \Phi^D \right)^{\Vert} \vert = O(r^\nu)  \ , \
\vert \left( \Phi - \Phi^D \right)^{\perp} \vert = O(e^{-r}) \ , \
\vert A - A^D \vert = O(e^{-r}) 
$$
on $\mathbb{R}^3 \backslash \lbrace 0 \rbrace$. Moreover, if $k=1$ i.e. for the BPS monopole then $\nu = - \infty$, i.e. $\vert  \Phi - \Phi^D \vert = O(e^{-r})$.
\end{lemma}

\section{Dirac Monopoles}\label{section:DiracMonopole}

This section contains the linear analysis necessary for the construction of Dirac monopoles, the main result comes in the form of theorem \ref{DiracMonopole}, below.\\
Let $l=b_0(\Sigma)$, then given a total charge $k \in \mathbb{N}$ and mass $m=(m_1, ...,m_{l}) \in  \mathbb{R}^{l} \cong H^0(\Sigma, \mathbb{R})$ we shall construct Dirac monopoles with this mass and total charge. The quantity
$$m_0(m) = \frac{1}{l} \sum_{i=1}^{b_0(\Sigma)} m_i, $$
is called the average mass of the monopole.

\begin{definition}
Let $p=(p_1,...,p_k) \in X^k$ and define the current $\delta \in (C^{\infty}_0(X))^*$, by
$$\delta(f)= \sum_{i=1}^{k} f(p_i) $$
for a compactly supported $f \in C^{\infty}_0(X)$.
\end{definition}

These points can be repeated and so the current $\delta$ has multiplicities giving a vector $k^{I}=(k_1^{I}, ..., k_{\sharp-pts}^{I}) \in \mathbb{Z}^{\sharp-pts}$ of charges generating a flux that must then leave $X$ through its ends with some charges $k^O \in \mathbb{Z}^{l}$ with $\vert k^I \vert = \vert k^O \vert$.
The Laplacian $\Delta$ acts $C^{\infty}_0(X)$ and one can consider its transpose operator, also denoted by $\Delta$ acting on $H \in (C^{\infty}_0)^*$ by $\Delta H (g) = H(\Delta g)$, for all $g \in C^{\infty}_0(X)$.

\begin{proposition}\label{LinearProposition}
Let $(p_1,...,p_k) \in X^k$ and $(m_1,...,m_{l}) \in  \mathbb{R}^{l}$, then there is a current $H^D \in (C^{\infty}_0(X))^*$ such that $\Delta H^D = \delta$. This can be represented by an integral operator
$$H^D(f) = \int_X f \phi^D ,$$
where $\phi^D$ is an harmonic function on $X \backslash \lbrace p_1,...,p_k \rbrace$ and $k^O \in \mathbb{R}^{l}$ satisfying $\vert k^O \vert = \vert k^I \vert$ such that
\begin{eqnarray}\label{BehaviourAtPoints}
\phi_D \big\vert_{U(p_i)} & = & - \frac{k_i^{I}}{2r} + O(r^0)  \\ \label{BehaviourAtEnds}
\phi_D \big\vert_{U(\Sigma_i)} & = & m_i - \frac{k_i^{O}}{Vol(\Sigma_i)} \frac{1}{r} + O(r^{-2}). 
\end{eqnarray}
Where $U(p_i)$ and $U(\Sigma_i)$ respectively denote a neighborhood of $p_i$ and the $i$-th end.
\end{proposition}

The proof is an exercise in the calculus of variations, which below is hidden by the use of the Riesz representation theorem. Before the proof two lemmas are required. Let $V$ denote the space of functions $f \in L^1_{loc}(X)$, with $\nabla f \in L^2$. Given $g \in L^1_{loc} $, a function $u \in V$ is called a weak solution to $\Delta u = g$ if for all smooth compactly supported $\psi$ 
$$\langle du , d \psi \rangle_{L^2} = \int_X g \psi. $$
Recall that if $f: X \rightarrow \mathbb{R}$ is such that $\nabla f \in L^2$, then according to lemma \ref{LimitLemma} the uniform limit of $f \vert_{\varphi^{-1} ( \lbrace r \rbrace \times \Sigma_i )}$ along any end exists, is constant and will be denoted by $f\vert_{\Sigma_i}$ on each connected component $\Sigma_i$ of $\Sigma$.

\begin{lemma}\label{LinearProblem1}
Let $g \in L^{\frac{6}{5}}(X)$, then there is a unique weak solution $u_g \in V$ of 
\begin{eqnarray}\nonumber
\Delta u_g & = & g \\ \nonumber
u_g \vert_{\Sigma} & = & 0.
\end{eqnarray}
\end{lemma}
\begin{proof}
Those $f \in V$ such that $f\vert_{\Sigma_i}$ vanishes form an Hilbert space $H$, namely the completion of the smooth compactly supported functions in the inner product $\langle f_1 , f_2 \rangle_H = \int_X df_1 \wedge \ast df_2$, for $f_1 , f_2 \in C^{\infty}_0(X)$. To find a weak solution $u_g \in H$ one proves the linear functional $f \mapsto \int_X g f$ is bounded on $H$. This follows from
$$\vert \int_X g f \vert \leq \Vert g \Vert_{L^{\frac{6}{5}}} \Vert f \Vert_{L^6} \leq \Vert g \Vert_{L^{\frac{6}{5}}} \Vert f \Vert_{H},$$
where we used first H\"older's inequality and then the Sobolev inequality from lemma \ref{Kato}. The Riesz representation theorem gives an element $u_g \in H$ such that $\langle u_g , f \rangle_H = \int_X gf$, for all $f \in H$, i.e. $u_g$ is a weak solution to the problem. To prove uniqueness suppose there are two solutions $u,v \in H$. Then $w= u-v$ is an harmonic function in $H$. Moreover, since $w \in H$ we have $dw \in L^2$ and $w \vert_{\Sigma}=0$ and so there is a sequence $R_i \rightarrow \infty$ such that $ w \vert_{r^{-1}(R_i)} = o(R_i^{-1/2})$ and $dw \vert_{R_i}=o(R_i^{-3/2})$. So one can compute
\begin{eqnarray}\nonumber
0 & = & \int_X w \Delta w = - \int_X  \left(d(w\ast dw) - dw \wedge \ast dw \right) \\ \nonumber
& = & \lim_{i \rightarrow \infty} \int_{r^{-1}(R_i)} w \ast dw + \Vert w \Vert_{H}^2 \\ \nonumber
& = & \Vert w \Vert_{H}^2 ,
\end{eqnarray}
and conclude $w \in H$ with $dw=0$ and so $w=0$, i.e $u=v$.
\end{proof}

\begin{corollary}\label{LinearProblem2}
Given $m = (m_1 ,..., m_{l})\in \mathbb{R}^{l}$, there is a unique solution $u_m \in V$ of the problem
\begin{eqnarray}\nonumber
\Delta u_m & = & 0 \\ \nonumber
u_m \vert_{\Sigma_i} & = & m_i \ , \ i=1,..., l.
\end{eqnarray}
\end{corollary}
\begin{proof}
Let $g \in C^{\infty}(X)$ with $g \vert_{\Sigma} =m$ be a smooth extension of $m$ to the whole $X$, such that $\Delta g \in L^{\frac{6}{5}}$. Then from lemma \ref{LinearProblem1} one concludes that there is a unique solution $\tilde{u} \in H$ to the problem $\Delta \tilde{u}= \Delta g$. Then the solution $u_m$ is given by setting $u_m = g- \tilde{u}$.
\end{proof}

\begin{corollary}\label{LinearProblem3}
Let $g \in L^{\frac{6}{5}}(X)$ and $m = (m_1 ,..., m_{l})\in \mathbb{R}^{l}$, there is a unique solution $u \in V$ of the problem
\begin{eqnarray}\nonumber
\Delta u & = & g \\ \nonumber
u \vert_{\Sigma_i} & = & m_i \ , \ i=1,..., l.
\end{eqnarray}
\end{corollary}
\begin{proof}
Let $u_g \in H \subset V$ be the solution to $\Delta u_g = g$, $u_g \vert_{\Sigma}=0$ given by lemma \ref{LinearProblem1} and let $u_m \in V$ be the solution to $\Delta u_m =0$, $u_m \vert_{\Sigma} =m$ given by corollary \ref{LinearProblem2}. Then $u=u_g + u_m$ is the desired solution and the uniqueness follows from a similar argument as the one used in the proof of lemma \ref{LinearProblem1}.
\end{proof}

\begin{proof}(of proposition \ref{LinearProposition})
There is a $1$-parameter family of smoothings of the current $\delta$ represented by smooth $3$-forms $\delta^{\epsilon}dvol$, such that $\delta(f) = \lim_{\epsilon \rightarrow 0} \int_X \delta^{\epsilon} f$, for all $f \in C^{\infty}_0(X)$. The $\delta^{\epsilon}$ can be chosen to have uniformly bounded $L^1$-norm, i.e. $\Vert \delta^{\epsilon} \Vert_{L^1} = \int_X \vert \delta^{\epsilon} \vert dvol_X =k$, and to be supported on small $\epsilon >0$ balls around the points $p_i$. At this point it must be remarked that the original distribution $\delta$ does not make sense as an element of $H^*$, as elements of $H$ need not be bounded. However, it does makes sense as a current, i.e. $\delta \in (C^{\infty}_0(X))^*$ as $\Vert \delta^{\epsilon} \Vert_{L^1} =k$ is bounded independently of $\epsilon$ and so
$$ \vert \delta(f) \vert = \Big\vert \lim_{\epsilon \rightarrow 0} \int_X \delta^{\epsilon} f \Big\vert \leq k \Vert f \Vert_{L^{\infty}}.$$
The trick now is to understand that for each $\epsilon$ the norm $\Vert \delta^{\epsilon} \Vert_{L^{\frac{6}{5}}}$ is still bounded, however not independently of $\epsilon$. Then corollary \ref{LinearProblem3} gives a family of functions $\phi_{D}^{\epsilon} \in V$, weakly solving
\begin{eqnarray}\nonumber
\Delta \phi_{D}^{\epsilon} & = & \delta^{\epsilon} \\ \nonumber
\phi_{D}^{\epsilon} \vert_{\Sigma_i} & = & m_i \ , \ i=1,..., l,
\end{eqnarray}
and with $\phi_D^{\epsilon}$ unique for each $\epsilon$. Since the $\delta^{\epsilon}$ are smooth, elliptic regularity guarantees that so are the $\phi_{D}^{\epsilon}$. However it must be remarked that the norm $\Vert \phi_{D}^{\epsilon} \Vert_H$ is not uniformly bounded independently of $\epsilon$. For $f \in C^{\infty}_0(X)$
$$\delta(f)= \lim_{\epsilon \rightarrow 0} \int_X \delta^{\epsilon} f = \lim_{\epsilon \rightarrow 0} \langle \phi_{D}^{\epsilon} , f \rangle_H,$$
and since for all $\epsilon$ we have $ \langle \phi_{D}^{\epsilon} , f \rangle_H =\int_X \delta^{\epsilon} f \leq k\Vert f \Vert$, the weak limit as $\epsilon \rightarrow 0$ of the $\phi^D_{\epsilon}$ exists and gives a current $\phi^D$ weakly solving $\Delta \phi^D = \Delta \delta$. This current is represented by an unbounded function which we still denote by $\phi^D$ such that the integral $\int_X f \phi^D=\lim_{\epsilon \rightarrow 0} \int_X \phi_{D}^{\epsilon} f$ is well defined for all $f \in C^{\infty}_0(X)$. As the $L^1$-norm of the $\delta^{\epsilon}$ is bounded independently of $\epsilon$ and $\Delta \phi_D^{\epsilon}=0$ outside an $\epsilon$-neighborhood of the $p_i$'s one respectively concludes that $\phi^D \in L^1_{loc}$ and is smooth away from the $p_i$'s.\\
Moreover, as the metric is asymptotically conical $\phi_D$ behaves as \ref{BehaviourAtEnds} at each end. Locally on small balls $U_{p_i}$ around each $p_i$ the metric is approximately Euclidean, so that on these \ref{BehaviourAtPoints} holds. The last thing to be checked is the identity $\vert k^I \vert = \vert k^O \vert$, this follows from integrating
\begin{eqnarray}\nonumber
0 & = & \int_{X \backslash \lbrace p_1,...,p_{\sharp-pts.} \rbrace} \Delta \phi^D dvol = - \int_{X \backslash \lbrace p_1,...,p_{\sharp-pts.} \rbrace} d\ast d \phi^D \\ \nonumber
& = & \lim_{\epsilon \rightarrow 0} \sum_{i=1}^{\sharp-pts.} \int_{\partial B_{\epsilon}(p_i)}\ast d \phi^D  - \lim_{r \rightarrow \infty} \sum_{i=1}^{\sharp-ends.} \int_{\Sigma_{i}^r}\ast d \phi^D  \\ \nonumber
& = &  \sum_{i=1}^{\sharp-pts.} k^I_i  -  \sum_{i=1}^{\sharp-ends.} k^O_i,
\end{eqnarray}
where we use Stokes' theorem and the local behavior of $\phi^D$ both at the singular points $p_i$ and at the ends $\Sigma_i$.
\end{proof}

\begin{remark}
The distribution $\delta$ can be extended from the smooth compactly supported functions to those $f$ which are smooth and bounded. Then, the second Green's identity gives
\begin{equation*}
\delta (f) = \left\{
\begin{array}{rl}
&  \int_X  d \phi_D \wedge \ast d f , \\
&  \int_X  f \Delta \phi_D + \int_{\Sigma} f \ast d\phi_D ,
\end{array} \right.
\end{equation*}
where the integrals involving $\phi^D$ can be interpreted as the corresponding limits as $\epsilon \rightarrow 0$. Since, $\ast d \phi_D \vert_{\Sigma_i} = \frac{k^O_i}{Vol(\Sigma_i)}dvol_{\Sigma_i}$,
$$\int_{\Sigma} f \ast d\phi_D= \sum_{i=1}^{b_0(\Sigma)} k^O_i  \frac{1}{Vol(\Sigma_i)} \int_{\Sigma_i}f ,$$
moreover if $df \in L^2$, then $f$ converges to a constant by lemma \ref{LimitLemma} and the formula above simplifies to $\int_{\Sigma} f \ast d\phi_D= \sum_{i=1}^{b_0(\Sigma)} k_i^O f \vert_{\Sigma_i}$. One still needs to show $\Delta \phi_D = \sum_{i=1}^{\sharp-pts.} \delta(p_i)$, but this can be achieved by the limit $\lim_{\epsilon \rightarrow 0} \Delta \phi^{\epsilon}_D$ this can be seen to converge uniformly to zero outside the $p_i$'s leaving $\sum_{i=1}^{\sharp-pts.} \delta_{p_i}$. To check that $\Delta \phi_D$ does not get any contribution at $\infty$ one needs to use the fact that the metric is asymptotically conical and so at the $i$-th end
$$\phi^{\epsilon}_D = m_i -  \frac{k^O_i}{Vol(\Sigma_i)} \frac{1}{\sqrt{r^2 + \epsilon^2}} + o(r^{-2}),$$
one computes
\begin{eqnarray}
\Delta \phi^{\epsilon}_D & = & \frac{\partial^2  \phi^{\epsilon}_D}{\partial r^2} + \frac{2}{r} \frac{\partial \phi^{\epsilon}_D}{\partial r} + ... \\ \nonumber
& = & \frac{ k^O_i }{ Vol(\Sigma_i) } \frac{3\epsilon^2}{ (r^2 + \epsilon^2)^{\frac{5}{2}} } + ...
\end{eqnarray}
and so the integration of $\Delta \phi^{\epsilon}_D$ times a bounded function is finite and then converges to zero as $\epsilon \rightarrow 0$. Inserting all this information one concludes that $\delta$ extends as
\begin{equation*}
\delta (f) =  \sum_{i=1}^{\sharp-pts} k^{I}_i f(p_i) - \sum_{i=1}^{\sharp-ends} k^O_i  \frac{1}{Vol(\Sigma_i)} \int_{\Sigma_i}f.
\end{equation*}
\end{remark}

The next goal is to construct a line bundle with connection on $U = X \backslash \lbrace p_i \rbrace_{i=1}^{\sharp-pts.}$

\begin{proposition}\label{2Form}
Let $U = X \backslash \lbrace p_i \rbrace_{i=1}^{\sharp-pts.}$ and $F_D = \ast d \phi_D$. Then $F_D$ is a closed $2$-form on $U$ and the class $[F_D] \in H^2(U, \mathbb{R/Z})$ is uniquely determined by a class $[F_X] \in  H^2(X, \mathbb{R/Z})$. Moreover, the $F_D$ is the curvature of a connection on a line bundle over $U$ if and only if the class $[F_X]$ vanishes.
\end{proposition}
\begin{proof}
Let $R$ be one of the following Abelian Groups $\mathbb{Z}, \mathbb{R}, \mathbb{R/Z}$. Since $H^1(U_{p_i} \backslash \lbrace p_i \rbrace)$, $H^2(U_{p_i})$ and $H^3(X)$ all vanish, the long exact sequence for the pair $U=X \backslash \lbrace p_i \rbrace_{i=1}^{\sharp-pts}$, $\cup_{i=1}^{\sharp-pts} U_{p_i}$, gives
$$0 \rightarrow H^2(X,R) \rightarrow H^2(U,R) \rightarrow \bigoplus_{i=1}^{\sharp-pts} H^2(U_{p_i} \backslash \lbrace p_i \rbrace , R) \rightarrow 0.$$
The $2$-form $F_D =  \ast d \phi_D$ is closed in $U$, since $dF_D = -\ast \Delta \phi^D$ which vanishes in $U$. It represents a class in $H^2(U,\mathbb{R})$ and is the curvature of a connection on a line bundle over $U$ if and only if this class has integer periods, equivalently if $[F_D]$ vanishes in $H^2(U,\mathbb{R/Z})$. In fact, from the definition of $\phi_D$ one knows that the image of $[F_D]$ in $H^2(U_{p_i} \backslash \lbrace p_i \rbrace, \mathbb{R/Z})$ vanishes, since
$$\int_{\partial \overline{U}_{p_i}} F_D = \int_{U_{p_i}} dF_D = \int_{U_{p_i}} \delta = k_i^{I} \in \mathbb{Z}.$$
So by exactness of the Mayer-Vietoris sequence, the class $[F_D] \in H^2(U, \mathbb{R/Z})$ is the image of a class in $[F_X] \in H^2(X, \mathbb{R/Z})$ and for $F_D$ to be the curvature of a connection on a line bundle over $U$ one just needs $[F_X]$ to vanish as well. 
\end{proof}

\begin{definition}\label{bMap}
\begin{enumerate}
\item Define the action of $R$ on $H^{0}(\Sigma, R) \cong R^{b_0(\Sigma)}$ given by
\begin{equation}\label{eq:Raction}
c.(m_1,...,m_{b_0(\Sigma)}) = (m_1+c,...,m_{b_0(\Sigma)}+c),
\end{equation}
for $c \in R$ and let $H^{0}(\Sigma, R)/ R$ denote the space of $R$-equivalence classes.

\item Define the map 
\begin{eqnarray}
b : X^{k} \times H^{0}(\Sigma, \mathbb{R})/ \mathbb{R} & \rightarrow & H^2(X, \mathbb{R/Z}) \\ \nonumber
(p, [m]) & \mapsto & [F_X],
\end{eqnarray}
where $[F_X]$ is the class determined via proposition \ref{2Form}. Moreover, fixing the mass $m$ denote by $b_{[m]}$ the map that to $p \in X^k$ assigns $b_{[m]}(p)= b(p, [m]) \in H^2(X, \mathbb{R/Z})$.
\end{enumerate}
\end{definition}

The map $b$ above is indeed well defined, since changing the mass 
$$m=(m_1,...,m_{b_0(\Sigma)})$$
by an overall constant amounts to add this constant to $\phi_D$. Hence $F_D= \ast d\phi_D$ remains unchanged and proposition \ref{2Form} can be rephrased as

\begin{corollary}\label{cor:mapb}
Given $(p, [m]) \in X^k \times H^0(\Sigma, \mathbb{R})/ \mathbb{R}$, these determine a line bundle $L$ over $U = X \backslash \lbrace p_1 , ..., p_k \rbrace $ with connection $\nabla_D$ and curvature $F_D= \ast d \phi_D$ if and only if $b(p, [m]) =0$. 
\end{corollary}

\begin{remark}
To make a connection with Hitchin's point of view in \cite{Hitchin99} one must regard the formal sum of points $\sum_{i=1}^k p_i$ as a sort of divisor on the $3$ dimensional Riemannian manifold $(X,g)$. This divisor can be equivalently regarded as a flat gerbe, since $H^3(X, \mathbb{Z})=0$ as $X$ is noncompact. The holonomy of the flat connection represents a class in $H^2(X, \mathbb{R/Z})$ and then one has the following equivalent statements. $1.$ The holonomy vanishes. $2.$ The flat gerbe is trivial. $3$ The divisor is linearly trivial. $4$ There is a Dirac monopole on $U$.
\end{remark}

\begin{theorem}\label{DiracMonopole}
Let $m \in H^{0}(\Sigma, \mathbb{R})$, then for each 
$$(p, \alpha) \in  b_{[m]}^{-1}(0)  \times H^{1}( X, \mathbb{S}^1)$$
there is a reducible $SU(2)$ bundle $L \oplus L^{-1}$ over $X \backslash \lbrace p_1,...,p_{k} \rbrace$ equipped with a charge $k$ and mass $m$ Dirac monopole $(A_D, \Phi_D)$. In particular, if $b_2 =0$, then the map $b$ vanishes and there is such a Dirac monopole for all $p \in X^k$.\\
Moreover, for any two different $(p, \alpha)$ and $(p', \alpha')$ the Dirac monopoles associated with these are not related by any gauge transformation $g \in C^{\infty}(X, \mathbb{S}^1)$.
\end{theorem}
\begin{proof} This has $4$ steps.
\begin{enumerate}
\item From proposition \ref{LinearProposition} one constructs the function $\phi_D$, which determines a $2$ form via $F_D=\ast d \phi_D$. Then, from proposition \ref{2Form} or equivalently corollary \ref{cor:mapb} one knows that the two form $F_D$ is the curvature of a connection $\nabla_D$ on a line bundle $L$ over $X \backslash \lbrace p_1,...,p_{k} \rbrace$ if and only if $b(p,[m])=0$.\\
We also note that the connections $\nabla_D$ associated with different tuples of points are obviously gauge inequivalent and may possibly live in different line bundles.
%
%\item For each representative $c=(c_1,...,c_{b_0(\Sigma)}) \in H^{0}(\Sigma, \mathbb{R})$ of the class $[c] \in  H^{0}(\Sigma, \mathbb{S}^1)/ \mathbb{S}^1$, corollary \ref{LinearProblem2} gives a unique harmonic function $f_c$ which converges at the ends $\Sigma_i$ to $c_i$. Moreover, different representatives $c$ give rise to harmonic functions $f_c$ which differ by a constant, so the one form $a=2 \pi i df_c$ is uniquely determined by the class $[c] \in H^{0}(\Sigma, \mathbb{R})/ \mathbb{R}$. Then, we upgrade the connection $\nabla_D$ to $\nabla_D + a$, its curvature is
%$$F_{\nabla_D + a} = F_{D}+ da = F_{D},$$
%since $da=d^2f_c=0$. This shows the pair $(\nabla_D + a , \phi_D)$ is still a monopole on $L \rightarrow X \backslash \lbrace p_1,..., p_k \rbrace$.\\
%Note that the connections $\nabla_D$ an $\nabla_D + a$ are related by the gauge transformation $g=e^{if_c}$, which only converges to $1$ if the class $[c] \in H^{0}(\Sigma, \mathbb{R})$ vanishes. In order to ease notation in what follows the monopole $(\nabla_D + a , \phi_D)$ will still be denoted by $(\nabla_D, \phi_D)$.

\item A class $\alpha \in H^1(X, \mathbb{S}^1)$ represents a gauge equivalence class of flat connections $\nabla^F$ on a torsion line bundle $L^F$ over $X$. One can now twist this with the Dirac monopole without changing the Higgs field. This gives the line bundle $L \otimes L^F$, which one still denotes by $L$, equipped with the connection $\nabla = \nabla_D \otimes 1 + 1 \otimes \nabla^F$. The curvature of this connection remains unchanged and so it still satisfies the Bogomolnyi equation $\ast F_D = d\phi_D$.

\item To finish the proof increase the gauge group to $SU(2)$, either by equipping $L$ with a metric and considering the principal bundle $P= \mathbb{S}^1(L) \times_{U(1)} SU(2)$, or by considering the vector bundle $L \oplus L^{-1}$, associated with $P$ through the standard representation of $SU(2)$ in $\mathbb{C}^2$. Then, the adjoint bundle is $\mathfrak{g}_P = \underline{\mathbb{R}} \oplus L^2$. Define the connection $A_D = \nabla_D \oplus (-\nabla_D)$ on $E$ and Higgs field $\Phi_D = \phi_D \oplus 0$, these do satisfy the monopole equations on $U$ and have Dirac type singularities at the points $p_i$, for $i=1,...,k$.
\end{enumerate}
\end{proof}

\begin{remark}
The automorphism group of the pair $(A_D, \Phi_{D})$ constructed above on $X \backslash \cup_{i=1}^k \lbrace p_i \rbrace$ is isomorphic to $\mathbb{S}^1$. In fact, given $e^{i a} \in \mathbb{S}^1$, we have $e^{i a \Phi_D} \cdot (A_D,\Phi_D)=(A_D, \Phi_D)$.
\end{remark}

\subsection{The Approximate Solution}\label{ApproximateSolution}

This section constructs an approximate solution to the monopole equations in proposition \ref{prop:AproximateSolution}, whose error term is estimated in lemma \ref{prop:ErrorEstimate}. Recall the map $b$ from definition in \ref{bMap} and theorem \ref{DiracMonopole} which proves that if $p=(p_1,...,p_{k}) \in X^{k}$ and $m \in H^{0}(\Sigma, \mathbb{R})$ are such that $b(p,[m])=0$, then one can construct a reducible Dirac Monopole $(A_D,\Phi_D)$ on an $SU(2)$-bundle $L \oplus L^{-1}$ over $X \backslash \lbrace p_1,...,p_{k} \rbrace$. Before proceeding to the construction of an approximate solution it is helpful to show how one can make this Dirac monopole have very large Higgs field in a big open set in $X$.

\begin{lemma}\label{MinimumDirac}
Let $[m] \in H^0(\Sigma, \mathbb{R})/ \mathbb{R}$ be an $\mathbb{R}$-equivalence class. Then, there is $\mu \in \mathbb{R}$ (depending on the metric $g$ and the class $[m]$), with the following property. Let $\phi_D$ the Higgs field constructed in proposition \ref{LinearProposition} associated with a representative $m \in [m]$ with average mass $m_0(m) \geq \mu$. Then, there is $\epsilon = \sqrt{ \frac{10}{8m_0}}$ 
$$ \phi^D \geq \frac{m_0}{2},$$
on $U_{\epsilon} = X \backslash \cup_{i=1}^{\sharp-pts.} \overline{B_{\epsilon} (p_i)}$.
\end{lemma}
\begin{proof}
Let $m_{init} \in H^0(\Sigma, \mathbb{R})$ be a representative in the class $[m]$ with average mass $m_0(m_{init})=0$. Then,equation \ref{BehaviourAtPoints} in roposition \ref{LinearProposition} implies that there are $c_i \in \mathbb{R}$ and a positive $\delta < inj(X,g)$ such that on $B_{\delta}(p_i)$, 
$$\phi_D(r)  =  c_i - \frac{1}{2r} + O(r).$$
One can always add a constant $m_0$ and construct a new Dirac monopole $\phi_D$ with mass $m= m_{init} + \left( m_0, ...,m_0 \right)$. Then $[m]=[m_{init}]$, but $m_0([m])=m_0$. Then, by possibly changing the constants $c_i$ by some term linear in $\delta$ one can assume that on each $B_{\delta}(p_i)$
$$\phi_D  \geq  m_0 + c_i - \frac{1}{2r},$$
and is increasing along the radial coordinate. The goal now is to increase $m_0$ in order to find $\epsilon < \delta$, such that for $r \in (\epsilon, \delta)$ one has $\phi_D \geq \frac{m_0}{2}$. As the $c_i$'s do not change with $m_0$, one can change $m_0$ so that $m_0 + c_i \geq \frac{9}{10} m_0$, i.e. $m_0 \geq 10 \max_i (- c_i )$. Then it is enough to solve
$$ \frac{9}{10}m_0 - \frac{1}{2r} \geq \frac{m_0}{2}, $$
which gives $r \geq \frac{10}{8m_0} $. This quantity is strictly decreasing with increasing $m_0$, and so we can arrange for $\sqrt{\frac{10}{8m_0}} << \delta$, as $\delta$ does not depend on $m_0$. Then, set $\epsilon = \sqrt{\frac{10}{8m_0}} > \frac{10}{8m_0}$ for big $m_0$. Then, one has 
$$\phi_D \vert_{\partial B_{\epsilon}} \geq  \frac{m_0}{2},$$
so that the problem is reduced to show that $\phi_D \geq \frac{m_0}{2}$ on the whole of $U_{\epsilon}= X \backslash \cup_i B_{\epsilon}(p_i)$. This follows from the fact that $\phi_D$ is harmonic on $U_{\epsilon}$ and so by the maximum principle it has no interior maximum or minimum. These are attained at the ends of $X$ or at the spheres $\partial B_{\epsilon}$. Since, the $\epsilon$'s where chosen smaller than $\delta$ and on each $B_{\delta}(p_i)$, $\phi_D$ is increasing with $r=dist(p_i, \cdot )$, the conclusion is that the minimum is attained at the inner boundaries and so $\phi_D \geq \frac{m_0}{2}$ on $X \backslash \cup_i B_{\epsilon}(p_i)$.
\end{proof}

\begin{definition}
Let $m \in H^{0}(\Sigma, \mathbb{R})$ be a configuration with positive average mass $m_0$. Define $X^k(m)$ as the set of $p=(p_1,...,p_k) \in X^k$ such that $b(p,[m])=0$ and $dist(p_i,p_j) \geq \frac{4}{\sqrt{m_0}}$, for all distinct $i,j \in \lbrace 1, ..., k \rbrace$.
\end{definition}

\begin{remark}
In the case where $b^2(X)=0$, the map $b$ vanishes identically and $H^0(\Sigma, \mathbb{R}) \cong \mathbb{R}^l$ so $X^k(m)$ is the open subset of $X^k$ where the distinct points are at distance more than $4/\sqrt{m}$ apart. However, in general this is the intersection of this open set $U(m)$ with a submanifold $b_{[m]}^{-1}(0)$ of $X^k$ of dimension $3k - b^2(X)$. Moreover, this submanifold only depends on the equivalence class $[m]$ and so by increasing $m_0(m)$ one can make the open set $U(m)$ as large as one wants and so $X^k(m)= b_{[m]}^{-1}(0) \cap U(m)$ is certainly nonempty.
\end{remark}

\begin{proposition}\label{prop:AproximateSolution}
Fix an identification $\mathbb{S}^1 \cong (\mathring{M}_{1}(\mathbb{R}^3))$ with the moduli space of charge $1$ centered monopoles on $\mathbb{R}^3$. Denote $k$ copies of this by $\mathbb{T}^k$ and let $\check{\mathbb{T}}^{k-1} = \lbrace (e^{i\theta_1}, ... , e^{i \theta_k} ) \in \mathbb{T}^k \ \vert \ e^{i(\theta_1 + ... + \theta_k)}=1 \rbrace$. Then, for all $[m] \in H^0(\Sigma, \mathbb{R})/ \mathbb{R}$, there is $\mu_0>0$, such that for $m_0(m) > \mu_0$ there is a function
\begin{eqnarray}
H: X^k(m) \times H^1(X, \mathbb{S}^1) \times \check{\mathbb{T}}^{k-1} \rightarrow C_{k,m} (X,g),
\end{eqnarray}
with the property that for each configuration $(A_0,\Phi_0)= H(p, \alpha, \theta)$ there are real numbers $\epsilon^i_{out} = 2\epsilon^i_{in} =O(m_0^{-1/2})$ and $\lambda_i = O(m_0)$ such that 
\begin{itemize}
\item $(A_0, \Phi_0)$ is a configuration on a $SU(2)$ bundle $E$, such that $E \vert_{X \backslash \lbrace p_1,...,p_k \rbrace} \cong L \oplus L^{-1}$, where $L \oplus L^{-1}$ is the bundle from theorem \ref{DiracMonopole} and $\vert \Phi_0 \vert > \frac{m_0}{2}$.
\item On $X \backslash \cup_{i=1}^{\sharp-pts.} B_{\epsilon^i_{out}}(p_i)$, the configuration $(A_0,\Phi_0)$ coincides with a Dirac monopole,
\item Using geodesic normal coordinates on each $B_{\epsilon^i_{in}}(p_i)$, $(A_0,\Phi_0)$ coincides with a centered charge $1$ and mass $\lambda_i$ BPS monopole on $\mathbb{R}^3$.
\end{itemize}
\end{proposition}
\begin{proof}
Let $x=(p, \alpha , \theta) \in X^k(m) \times H^1(X, \mathbb{S}^1) \times \check{\mathbb{T}}^{k-1}$. Starting with $(p,\alpha)$, theorem \ref{DiracMonopole} in the last section constructs a reducible $SU(2)$ Dirac monopole $(A_D,\Phi_D)$ on $L \oplus L^{-1}$ over $X \backslash \lbrace p_1,..., p_k \rbrace$. From the proof of lemma \ref{MinimumDirac}, in a neighborhood of each $p_i$, one can write
$$\phi_D = -\frac{1}{2r} + c_i + m_0 + O(r),$$
where the $c_i$'s are constants independent of the average mass $m_0$ and depend only on the points $p_i$ and the class $[m]$. Moreover there is $\mu >0$, such that for $m_0 > \max \lbrace 1+ 2\sup_i \vert c_i \vert, \mu \rbrace$ and define
\begin{eqnarray}\label{lambdai}
\lambda_i & = & m_0 + c_i \\ \label{epsiloni}
\epsilon_{in}^i & = & \lambda_i^{-1/2},
\end{eqnarray}
and one can check that locally around each $p_i$ for $r> \epsilon_{in}^i$, it holds that $\phi_D > \frac{\lambda_i}{2}$.

The identification $\mathbb{S}^1 \cong \mathring{M}_{1}(\mathbb{R}^3)$ gives for each element of $ \theta_i \in \mathbb{S}^1$ a framed centered BPS monopole $(A^{BPS}, \Phi^{BPS} , \eta_i)$. As the bundle $L$ has degree one on small spheres around each $p_i$ we can use the framing $\eta_i $ to identify the trivial $\mathbb{C}^2$ bundle over $\partial B_{\epsilon^i_{out}} \backslash \lbrace p_i \rbrace$ with the restriction of $L \oplus L^{-1}$. This extends the bundle $L \oplus L^{-1}$ over the points $p_i$ to form the bundle $E$.\\
Recall from definition \ref{def:BPSandR} that $\vert \Phi^{BPS} \vert \geq \frac{1}{2}$, outside $B_1(0) \subset \mathbb{R}^3$. After scaling these BPS configurations to have mass $\lambda_i$, they satisfy $\vert \Phi^{BPS}_{\lambda_i} \vert \geq \frac{\lambda_i}{2}$ outside the ball of radius $\lambda_i^{-1}$ in $\mathbb{R}^3$. Let $\epsilon^i_{out} = 2(\lambda_i)^{-1/2}$, then $dist(p_i,p_j) > 2 \epsilon^i_{out}$. Then using geodesic normal coordinates centered at the points $p_i$ and $\theta = (\theta_1, ... , \theta_k ) \in  \check{\mathbb{T}}^{k-1}$ to identify the bundles we can pullback these scaled monopoles $(A^{BPS}_{\lambda_i}, \Phi^{BPS}_{\lambda_i})$ from a small ball $B_{\epsilon^i_{out}} \subset \mathbb{R}^3$ to the ball $B_{\epsilon^i_{out}}(p_i)$.

Now consider the open cover of $X$ given by the $B_{\epsilon^i_{out}}(p_i)$ and $U_{\epsilon}= X \backslash \cup_{i=1} \overline{ B_{\epsilon^i_{in}}(p_i)} $ and let $\lbrace \chi_{out}, \chi_1 , ..., \chi_{k} \rbrace$ be a partition of unity subordinate to this cover, such that
\begin{equation*}
\chi_{in}^i = \left\{
\begin{array}{rl}
1 & \text{in } B_{\epsilon^i_{in}}(p_i),\\
0 & \text{in } X \backslash B_{\epsilon^i_{out}}(p_i).
\end{array} \right. \ \ , \ \ 
\chi_{out} = \left\{
\begin{array}{rl}
0 & \text{in } \cup_{i=1}^{k} B_{\epsilon^i_{in}}(p_i),\\
1 & \text{in } X \backslash \cup_{i=1}^{k} B_{\epsilon^i_{out}}(p_i).
\end{array} \right. 
\end{equation*}
and
$$\sum_{i=1}^{k}\vert \nabla \chi_i \vert + \vert \nabla \chi_{out} \vert \leq c \sup_{i\in \lbrace 1,...,k \rbrace} \left( \epsilon^i_{out} - \epsilon^i_{in} \right)^{-1}\leq c \sqrt{m_0}.$$
Define the approximate solution as the following configuration on $E$
\begin{eqnarray}\label{Approximate1}
\Phi_0 & = & \chi_{out} \Phi_D + \sum_{i=1}^{k} \chi^i_{in} \Phi^{BPS}_{\lambda_i} \\  \label{Approximate2}
A_0 & = & A_D + \sum_{i=1}^{k} \chi^i_{in} (A^{BPS}_{\lambda_i} - A_D).
\end{eqnarray}
One must remark that even though the connection $A_D$ does not extend over the points $p_i$, the connection $A_0$ does, as on each trivialization over $B_{\epsilon^i_{in}}(p_i)$, the connection $A_0$ is represented by the connection $A^{BPS}_{\lambda_i}$.
\end{proof}

\begin{remark}
Recall that the automorphism group of the pair $(A_D, \Phi_{D})$ on $X \backslash \cup_{i=1}^k B_{\epsilon^i_{in}} (p_i)$ is isomorphic to $\mathbb{S}^1$ via $e^{i a} \mapsto e^{i a \Phi_D}$. Hence, given $\theta=(\theta_1, ..., \theta_k) \in \mathbb{T}^k$, the configuration $(A_0,\Phi_0)=H(p, \alpha , \theta)$ is gauge equivalent via $e^{i a \Phi_0}$ to $H(p,\alpha,\theta+ \vec{a})$ where $\theta + \vec{a}=(\theta_1 + a, ..., \theta_k + a )$.
\end{remark}

\begin{lemma}\label{prop:ErrorEstimate}
There is a constant $c>0$ independent of $m$, such that for any approximate solution $(A_0,\Phi_0)$ in the image of the map $H$ the quantity $e_0 = \ast F_{A_0} - \nabla_{A_0} \Phi_0$ vanishes outside the balls $B_{\epsilon^i_{out}}(p_i)$ and its $C^0$ norm is bounded by
$$\vert e_0 \vert  \leq c m_0 .$$ 
In particular, if the metric $g$ on $X$ is exactly Euclidean on the balls $B_{\epsilon^i_{out}}(p_i)$, then in fact the error $e_0$ is supported on the annulus $B_{\epsilon^i_{out}}(p_i) \backslash B_{\epsilon^i_{in}}(p_i)$.
\end{lemma}
\begin{proof}
On $X \backslash \cup_{i=1}^{k} B_{\epsilon^i_{out}}(p_i)$ the configuration $(A_0, \Phi_0)$ agrees with the Dirac monopole $(A_D, \Phi_D)$ and hence satisfies the monopole equations. So, the error term $e_0 = \ast F_{A_0} - \nabla_{A_0} \Phi_0$ is supported on the small balls $\cup_{i=1}^{k} B_{\epsilon^i_{out}(p_i)}$. On each of these balls, say $B_{\epsilon^i_{out}}(p_i)$ the bundle $E$ has a trivialization and $\chi^i_{in}= 1- \chi_{out}$, these may be used to write
\begin{eqnarray}\nonumber
\Phi_0 & = & \Phi^{BPS}_{\lambda_i} + \chi_{out}(\Phi_D - \Phi^{BPS}_{\lambda_i}) \\ \nonumber
A_0 & = &  A^{BPS}_{\lambda_i} + \chi_{out}(A_D - A^{BPS}_{\lambda_i}).
\end{eqnarray}
Since this agrees with $(A^{BPS}_{\lambda_i}, \Phi^{BPS}_{\lambda_i})$ in the ball $ B_{\epsilon^i_{in}(p_i)}$, the error term only depends on how far the metric $g$ on these balls differs from the Euclidean one. Computing the curvature and the covariant derivative of the Higgs field
\begin{eqnarray}\nonumber
F_{A_0} & = & F_{A^{BPS}_{\lambda_i}} + d_{A^{BPS}_{\lambda_i}} \chi_{out}(A_D - A^{BPS}_{\lambda_i}) + \chi_{out}^2 \left[ (A_D - A^{BPS}_{\lambda_i}) \wedge (A_D - A^{BPS}_{\lambda_i}) \right] \\ \nonumber
& = & F_{A^{BPS}_{\lambda_i}} + d\chi_{out} \wedge (A_D - A^{BPS}_{\lambda_i}) +  \chi_{out}d_{A^{BPS}_{\lambda_i}}(A_D - A^{BPS}_{\lambda_i}) \\ \nonumber
& & + \chi_{out}^2 \left[ (A_D - A^{BPS}_{\lambda_i}) \wedge (A_D - A^{BPS}_{\lambda_i}) \right], \\ \nonumber
\nabla_{A_0} \Phi_0 & = & \nabla_{A^{BPS}_{\lambda_i}} \Phi_0 + \chi_{out} \left[ (A_D - A^{BPS}_{\lambda_i}) , \Phi_0 \right] \\ \nonumber
& = & \nabla_{A^{BPS}_{\lambda_i}} \Phi^{BPS}_{\lambda_i} + d \chi_{out} (\Phi_D - \Phi^{BPS}_{\lambda_i} ) +  \chi_{out} \nabla_{A^{BPS}_{\lambda_i}}(\Phi_D - \Phi^{BPS}_{\lambda_i} ) \\ \nonumber
& & + \chi_{out} \left[ (A_D - A^{BPS}_{\lambda_i}) ,\Phi_0 \right],
\end{eqnarray}
one concludes that the error term is given by
\begin{eqnarray}\nonumber
\ast F_{A_0}- \nabla_{A_0} \Phi_0 & = & \ast F_{A^{BPS}_{\lambda_i}} - \nabla_{A^{BPS}_{\lambda_i}} \Phi_0  \\ \nonumber 
& & + \chi_{out} \left( \ast d_{A^{BPS}_{\lambda_i}}(A_D - A^{BPS}_{\lambda_i}) - \nabla_{A^{BPS}_{\lambda_i}}(\Phi_D - \Phi^{BPS}_{\lambda_i} ) \right) \\ \nonumber
& & + \ast  \left( d\chi_{out} \wedge (A_D - A^{BPS}_{\lambda_i}) \right) -  d \chi_{out} (\Phi_D - \Phi^{BPS}_{\lambda_i} ) \\ \nonumber
& & + \chi_{out}^2 \ \ast \left[ (A_D - A^{BPS}_{\lambda_i}) \wedge (A_D - A^{BPS}_{\lambda_i}) \right] \\ \nonumber
& & -  \chi_{out}\left[ (A_D - A^{BPS}_{\lambda_i}) ,\Phi_0 \right] 
\end{eqnarray}
The first of these terms would vanish if the metric $g$ is Euclidean in these balls. The terms $\vert A_D - A^{BPS}_{\lambda_i} \vert$ and $\vert \Phi_D - \Phi^{BPS}_{\lambda_i} \vert$ always appear multiplied by $\chi_{out}$ or $d\chi_{out}$ which are supported outside $B_{\epsilon_{in}}(p_i)$. So, one starts by estimating these on the annulus $B_{\epsilon_{out}}(p_i) \backslash B_{\epsilon_{in}}(p_i)$, where lemma \ref{lem:BPSestimate} gives
\begin{eqnarray}\nonumber
\Phi_{D} = \Phi^{BPS}_{\lambda_i} + O( \lambda_i e^{-\lambda_i r} ) \ , \ A_{D} = A^{BPS}_{\lambda_i} + O( \lambda_i r e^{-\lambda_i r}),
\end{eqnarray}
and since on these annulus $r \in (\epsilon^i_{in}, \epsilon^i_{out})$ and $\lambda_i=O(m_0)$, while $\epsilon^i_{out}= O(m_0^{-1/2})$
$$\vert A_D - A^{BPS}_{\lambda_i} \vert + \vert \Phi_D - \Phi^{BPS}_{\lambda_i} \vert = O( m_0 e^{- \sqrt{m_0}} ).$$
and this gives
\begin{eqnarray}\nonumber
\vert \ast  \left( d\chi_{out} \wedge (A_D - A^{BPS}_{\lambda_i}) \right) \vert + \vert  d \chi_{out} (\Phi_D - \Phi^{BPS}_{\lambda_i} ) \vert & = & O( m_0^{3/2} e^{- \sqrt{m_0}} ) \\ \nonumber
\vert \chi_{out}^2 (A_D - A^{BPS}_{\lambda_i})^2 \vert + \vert  \chi_{out} \left[ (A_D - A^{BPS}_{\lambda_i}) ,\Phi_0 \right] \vert & = & O( m_0 e^{- \sqrt{m_0}} )\\  \nonumber
 \chi_{out}^2 \left( \vert \ast d_{A^{BPS}_{\lambda_i}}(A_D - A^{BPS}_{\lambda_i})\vert  + \vert \nabla_{A^{BPS}_{\lambda_i}}(\Phi_D - \Phi^{BPS}_{\lambda_i} ) \vert \right) & = & O( m_0 e^{- \sqrt{m_0}} ) .
\end{eqnarray}
Back to evaluate the first term $\ast F_{A^{BPS}_{\lambda_i}} - \nabla_{A^{BPS}_{\lambda_i}} \Phi_0 $, one needs to compare the metric $g$ with the Euclidean metric $g_0$ on the balls $B_{\epsilon^i_{out}}(p_i)$. Let the $\lbrace x_i \rbrace_{i=1,2.3}$ denote geodesic normal coordinates centered at the point $p_i$ and recall that these are used to pullback the configuration $(A^{BPS}_{\lambda_i}, \Phi^{BPS}_{\lambda_i})$ from $\mathbb{R}^3$ to this small ball. In these coordinates $g= g_{0} + x_i x_j \gamma^{ij} + O(x^3)$ for some symmetric $2$-tensor $\gamma= x_i x_j\gamma^{ij}$. After a short computation expanding the formula $\omega \wedge \ast \omega = \vert \omega \vert_g^2 dvol_g$ for any $k$-form $\omega$, one concludes that
$$\ast_g \omega = \left( 1 + \frac{\gamma(\omega, \omega)}{\vert \omega \vert_{g_0}^2} - \frac{1}{6}Ric^{ij}x_i x_j \right) \ast_0 \omega + O( x^3).$$
Applying this to $\ast F_{A^{BPS}_{\lambda_i}}$ gives
\begin{eqnarray}\nonumber
\ast F_{A^{BPS}_{\lambda_i}} - \nabla_{A^{BPS}_{\lambda_i}} \Phi^{BPS}_{\lambda_i} & = & \ast_0 F_{A^{BPS}_{\lambda_i}} - \nabla_{A^{BPS}_{\lambda_i}} \Phi^{BPS}_{\lambda_i} + \\ \nonumber
& + & \left( \frac{ \gamma( F_{A^{BPS}_{\lambda_i}}, F_{A^{BPS}_{\lambda_i}}) }{ \vert F_{A^{BPS}_{\lambda_i}} \vert_{g_0}^2} - \frac{1}{6}Ric^{ij}x_i x_j \right) \ast_0 F_{A^{BPS}_{\lambda_i}},
\end{eqnarray}
the first term vanishes because $(A^{BPS}_{\lambda_i}, \Phi^{BPS}_{\lambda_i})$ is a monopole for the metric $g_0$. Regarding the second term it is $O( (\epsilon^i_{out})^2 \vert F_{A^{BPS}_{\lambda_i}} \vert_{g_0} )$, then since $\epsilon^i_{out}= O(m_0^{-1/2})$ and $\vert F_{A^{BPS}_{\lambda_i}} \vert_{g_0} = O(m_0^2)$, which gives
$$\vert \ast F_{A_{\lambda_i}} - \nabla_{A_{\lambda_i}} \Phi_{\lambda_i} \vert_{g}= O(m_0).$$
Summing all these terms for big $m_0$ (all the previous ones were lower order compared to this one), gives the result in the statement.
\end{proof}

This lemma points out one other "problem", there is no hope in controlling the error term in the $C^0$ norm provided by the metric $g$. This is related to the fact that there is no mass independent lower bound on the first eigenvalue of $d_2 d_2^*$ for the large mass BPS monopole. The idea to overcome this issue is to rescale all data inside the balls $B_{\epsilon_{out}^i}(p_i)$ and work there with the metric $g_{m_0^{-1}}=m_0^{2} \exp_{m_{0}^{-2}}^*g$. Then one can use the scaling identity
\begin{equation}
\vert \ast_{\epsilon} F_{A_{\epsilon \lambda_i}} - \nabla_{A_{\epsilon \lambda_i}} \Phi_{\epsilon \lambda_i} \vert_{g_{\epsilon}} = \epsilon^2 \vert \ast F_{A_{\lambda_i}} - \nabla_{A_{\lambda_i}} \Phi_{\lambda_i} \vert_{g},
\end{equation}
for all $\epsilon>0$. Applying this with $\epsilon= m_0^{-1}$ gives the following result

\begin{corollary}\label{cor:ErrorEstimate}
There is $c \in \mathbb{R}^+$, such that $e_0 = \ast F_{A_0} - \nabla_{A_0} \Phi_0$ vanishes on $X \backslash \cup_{i=1}^{k} B_{\epsilon^i_{out}}(p_i)$ and one the interior of these balls
$$\vert (e_0)_{m_0^{-1}} \vert_{g_{m_0^{-1}}}  = m_0^{-2} \vert e_0 \vert_{g}  \leq c m_0^{-1}.$$
\end{corollary}

Based on this, one can define function spaces where not only the error term will be small but the first eigenvalue of $d_2 d_2^*$ is bounded from below.

\section{Proof of the Main Theorem}\label{Proof}

Our goal is to apply a version of the Banach space Implicit Function theorem to solve the monopole equation. In order to do that we divide our proof into 3 major steps: 1. In subsection \ref{FunctionSpaces} we introduce suitable function spaces. 2. In subsection \ref{LinearEquation} we prove that using these function spaces the linearisation of the monopole equation $d_2$ is surjective. 3. Finally in subsection \ref{NonlinearEquation} we use the result from the second step in order to find a solution to the monopole equation. This also requires the spaces defined in the first step to give the initial approximate solution a small error term. In fact, we shall be able to control the error term in the approximate solutions in a uniform manner with respect to the mass of the approximate solution.

\subsection{Function Spaces}\label{FunctionSpaces}

This section introduces function spaces specially adapted to solve the monopole equation. To proceed with the definition of these Function spaces some preparation is needed. Let $(A_0, \Phi_0)$ be the approximate solution constructed in section \ref{ApproximateSolution}. Then on $U= X \backslash \cup_{i=1}^k B_{\epsilon^i_{out}}(p_i)$ the adjoint bundle $\mathfrak{g}_E$ splits as $\mathfrak{g}_E= \underline{\mathbb{R}} \oplus L^2$ and one writes a section $f=f^{\Vert} \oplus f^{\perp}$ according to this splitting. Let $\beta \in \mathbb{R}$ and $K \subset X$ containing $\cup_{i=1}^k B_{1}(p_i)$ and for each $n \in \mathbb{N}_0$, define $W_n: X \rightarrow \mathbb{R}$ to be smooth weight functions, interpolating between the values

\begin{equation*}
W_n= \left\{
\begin{array}{rl}
m_0^{-2+n}, & \text{in } \cup_{i=1}^k B_{\epsilon^i_{out}}(p_i),\\
1 , & \text{at } \cup_{i=1}^k \partial B_{1}(p_i),\\
\rho^{n-\beta- \frac{3}{2}} & \text{in } X \backslash K,
\end{array} \right.
\end{equation*}

Define for smooth compactly supported $f$, the norms

\begin{eqnarray}\nonumber
\Vert f \Vert_{H_{0,\beta}}^2 & = & \sum_{i=1}^{k} \int_{B_{1}(p_i)} \vert W_0 f \vert^2_{g} dvol_g \\  \label{FunctionSpaces1}
& &+ \int_{X \backslash \cup_{i=1}^k B_{1}(p_i)} \left( \vert W_0 f^{\Vert} \vert_g^2 + \vert f^{\perp} \vert^2_g \right) dvol_g .\\ \nonumber
\Vert f \Vert_{H_{n,\beta}}^2 & = & \sum_{j=0}^n  \sum_{i=1}^{k}  \int_{B_{1}(p_i)} \vert W_j \nabla_{A^{BPS}}^j f \vert^2_{g} dvol_g \\ \label{FunctionSpaces2}
& & + \sum_{j=0}^k \int_{X \backslash \cup_{i=1}^k B_{1}(p_i)}  \left( \vert W_j \nabla_{A_0}^j  f^{\Vert} \vert_g^2 + \vert ad_{\Phi_0}^{n-j} \left( \nabla_{A_0}^j f^{\perp} \right) \vert^2_g \right) dvol_g .
\end{eqnarray}

\begin{definition}\label{def:FunctionSpaces}
For $n \in \mathbb{N}_0$ and $\beta \in \mathbb{R}$, let $H_{n,\beta}$ be the Hilbert space completion of $C^{\infty}_0(X, (\Lambda^0 \oplus \Lambda^1)\otimes \mathfrak{g}_E)$ in the inner product obtained from polarizing the norms \ref{FunctionSpaces1} and \ref{FunctionSpaces2}.
\end{definition}

\begin{remark}
\begin{enumerate}
\item Recall the error $e_0$ of $(A_0,\Phi_0)$ is supported in the small balls $B_{\epsilon^i_{out}}(p_i)$ and to make it small in the $C^0$ norm we used an $\epsilon=m_0^{-1}$-rescaled norm $g_{\epsilon}= \epsilon^{-2} \exp^*_{\epsilon}g$, i.e. for $f \in \Gamma(X, \Lambda^l \otimes \mathfrak{g}_E)$ one needed to use $ \vert f_{\epsilon} \vert_{g_{\epsilon}}= \epsilon^2 \vert f \vert_g$, where $f_{\epsilon} = \exp^*_{\epsilon}f$ if $l=1$ and $f_{\epsilon} = \epsilon \exp^*_{\epsilon}f$ if $l=0$. Indeed, the functions $W_n$ are defined so that $W_0= m_0^{-2}$, on $\cup_{i=1}^k B_{\epsilon^i_{out}}(p_i)$.

\item Outside a big compact set $K$ the function spaces to be defined will coincide with the usual $L^2$ spaces for the $f^{\perp}$ component and the Lockhart-McOwen weighted spaces in the $f^{\Vert}$ component. These are suitable to order to find a bounded right inverse to the operator $d_2$.
\end{enumerate}
\end{remark}

\subsection{The Linear Equation}\label{LinearEquation}

The goal of this section is to solve the linear equation associated with the linearized monopole equation. Namely, one needs to find a range for $\beta \in \mathbb{R}$ such that if $g \in H_{\beta-1}$, then one can solve
$$d_2 f = g$$
for $f \in H_{1,\beta}$. In other words one needs to construct a right inverse $Q$ to the operator $d_2: H_{1,\beta} \rightarrow H_{\beta-1}$. This whole section is occupied with the proof of

\begin{proposition}\label{prop:LinearEquation}
Let $(X,g)$ be asymptotically conical with $b^2(X)=0$. Then there is $\mu>0$, such that if $m_0(m)> \mu$, the operator
$$d_2: H_{1,\beta} \rightarrow H_{0,\beta-1},$$
associated with any approximate solution $(A_0,\Phi_0)$ constructed in proposition \ref{ApproximateSolution} is surjective for all $\beta>-1$. In particular, there is a uniformly bounded right inverse $Q:H_{\beta-1} \rightarrow H_{1,\beta}$.
\end{proposition}

To prove this proposition we need to show that the cokernel of $d_2$ in $H_{0,\beta-1}$ vanishes. Using the $L^2$ inner product one can identify $H_{\beta-1}^* \cong H_{0,-\beta-2}$ and the cokernel of $d_2$ in $H_{0,\beta-1}$ with the kernel of $d_2^*$ in $H_{0,-\beta-2}$. The goal is to find a range of $\beta$ for which this vanishes. To prove this it enough to prove an inequality of the form
\begin{equation}\label{eq:MainInequality}
\Vert d_2^* f \Vert_{H_{0,\alpha-1}} \geq c \Vert f \Vert_{H_{0,\alpha}},
\end{equation}
for some $\alpha \in \mathbb{R},c>0$ and all $f \in H_{1,\alpha}$. Then, we set $\beta=-\alpha-2$ and construct $Q:H_{0,\beta-1} \rightarrow H_{1,\beta}$ to be such that $u=Qf$ is the unique solution of $d_2 u = f$ such that $u$ is $H_{1,\beta}$-orthogonal to the kernel of $d_2$ in $H_{1,\beta}$. To proceed with the proof, let $\chi$ be a bump function such that
\begin{equation*}\label{eq:BumpFunction}
\chi = \left\{
\begin{array}{rl}
0 & \text{in } \cup_{i=1}^{k} B_{\epsilon^i_{in}}(p_i),\\
1 & \text{in } U=X \backslash \cup_{i=1}^{k} \overline{B_{\epsilon^i_{out}}(p_i)},
\end{array} \right. 
\end{equation*}
and vanishing with all derivatives on $\partial B_{\epsilon^i_{in}}(p_i)$. One can further suppose that there is a constant $C>0$, such that $\vert d \chi \vert_g \leq C m_0^{1/2}$ on each annulus where $d\chi$ is supported. Hence its $L^2$ norm is bounded by $\Vert d\chi \Vert_{L^2} \leq C m_0^{-1/2}$, as the volume of the set containing the support of $d\chi$ is of order $(\epsilon^i_{out})^3=O(m_0^{-3/2})$. Then, for any section $f \in \Omega^1(X, \mathfrak{g}_P)$ one can write $f = (1-\chi)f + \chi f$, where the first term is supported inside the balls $B_{\epsilon^i_{out}}(p_i)$ and the second one is supported on $U$, the complement to their closure.

\begin{lemma}\label{lem:IntermediateInequality}
There is $\mu >0$ such that if $m_0(m) > \mu$ and $\alpha< -1$, there is a constant $c>0$, such that
$$\Vert d_2^* (\chi f) \Vert_{H_{0,\alpha-1}} \geq c \Vert \chi f \Vert_{H_{0,\alpha}},$$
for all $f \in H_{1,\alpha}$.
\end{lemma}
\begin{proof}
Recall that in the complement $U$, the approximate solution $(A_0,\Phi_0)$ gives a splitting $\mathfrak{g}_P \cong \underline{\mathbb{R}} \oplus L^2$, which is preserved by the operator $d_2$. Moreover, this is an orthogonal decomposition with respect to the $Ad$-invariant metric on $\mathfrak{g}_P$. So
$$\Vert d_2^* (\chi f) \Vert_{H_{0,\alpha-1}}^2 = \Vert d_2^* (\chi f^{\Vert}) \Vert_{H_{0,\alpha-1}}^2 + \Vert d_2^* (\chi f^{\perp}) \Vert_{H_{0,\alpha-1}}^2,$$
and it is enough to separately prove the inequality for each of these terms.
\begin{enumerate}
\item The longitudinal component $\chi f^{\Vert}$ is a one form, so $\Vert \chi f^{\Vert} \Vert_{H_{n,\alpha}}= \Vert \chi f^{\Vert} \Vert_{L^2_{n,\alpha}}$ and
$$d_2^*(\chi f^{\Vert})= (\ast d(\chi f^{\Vert}), - d^* (\chi f^{\Vert})).$$
Then using corollary \ref{rmk:LM} guarantees that $\Vert d_2^*(\chi f^{\Vert}) \Vert_{H_{0, \alpha -1}}^2 \geq c \Vert  \chi f^{\Vert} \Vert_{H_{1, \alpha}}^2$, for all $\alpha <-1$ and some $c$ not depending of $f$.

\item For the transverse component $\chi f^{\perp}$, $\Vert \chi f^{\perp} \Vert_{H_{0,\alpha}}= \Vert \chi f^{\perp} \Vert_{L^2}$ for all $\alpha \in \mathbb{R}$. So we need to consider
\begin{equation}
d_2^* (\chi f^{\perp}) = \left( \ast d_{A_0} (\chi f^{\perp}) + [\chi f^{\perp}, \Phi_0], - \nabla_{A_0} \chi f^{\perp}  \right),
\end{equation}
in usual $L^2$ spaces. Notice that $\chi f^{\perp}$ is supported in $U$, where the approximate solution $(A_0,\Phi_0)$ coincides with the Dirac monopole $(A_D, \Phi_D)$ (see the second bullet in \ref{prop:AproximateSolution}). Using the decomposition $\mathfrak{g}_P = \mathbb{R} \oplus L^2$, we have $\Phi_D=\phi_D \oplus 0$, and $\vert \phi_D \vert > \frac{m_0}{2}$. Then, integrating by parts together with the Weitzenb\"ok formula \ref{MonopoleW}, gives
\begin{eqnarray}\nonumber
\Vert d_2^* (\chi f^{\perp})  \Vert_{L^2}^2 & = & \langle \chi f^{\perp} , d_2 d_2^* (\chi f^{\perp}) \rangle_{L^2} \\ \nonumber
& = & \langle \chi f^{\perp} , \nabla_D^* \nabla_D (\chi f^{\perp})  + \phi_D^2 \chi f^{\perp} + Ric(\chi f^{\perp}) \rangle_{L^2} \\ \label{eq:Dirac2Sided}
& = &  \Vert \nabla_D (\chi f^{\perp}) \Vert_{L^2}^2 + \Vert \phi_D \chi f^{\perp} \Vert_{L^2}^2 + \langle \chi f^{\perp} ,  \chi Ric(f^{\perp}) \rangle_{L^2}.
\end{eqnarray}
Recall, from lemma \ref{MinimumDirac} that, $\phi_D$ is harmonic on $U$ and is bounded from below by its values at $\partial \overline{U}$, these being greater or equal to $\frac{m_0}{2}$. Then, $\phi_D \geq \frac{m_0}{2}$ on $U$ and
\begin{eqnarray}\nonumber
\Vert  d_2^* (\chi f^{\perp})   \Vert_{L^2}^2 & \geq & \Vert \nabla_D (\chi f^{\perp}) \Vert_{L^2}^2 + \left( \frac{m_0^2}{4} - sup_X \vert Ric \vert \right) \Vert \chi f^{\perp} \Vert_{L^2}^2 \\ \nonumber
& \geq & \Vert \nabla_D (\chi f^{\perp}) \Vert_{L^2}^2 + \frac{ \frac{m_0^2}{4} - sup_X \vert Ric \vert }{(m_0 + \max_i (m_i))^2 } \Vert \chi \phi_D f^{\perp} \Vert_{L^2}^2 \\ \nonumber
& \geq & c_2 \Vert \chi f^{\perp} \Vert_{H_{1, \alpha}}
\end{eqnarray}
for any $\alpha \in \mathbb{R}$ and one can take $c_2 = \min \big\lbrace 1 , \frac{ \frac{m_0^2}{4} - sup_X \vert Ric \vert  }{(m_0 + \max_i (m_i) )^2 } \big\rbrace >0$ for $m_0 > 2 \sqrt{\Vert Ric \Vert_{L^{\infty}}}$.
\end{enumerate}
Sum these two inequalities to get
$$\Vert d_2^* (\chi f) \Vert_{H_{0,\alpha-1}} \geq c \Vert \chi f \Vert_{H_{1,\alpha}},$$
with $c= \min \lbrace c_1 , c_2 \rbrace$. For the two inequalities to be true and the statement to hold one needs $\alpha < -1$ due to the component $\chi f^{\Vert}$, and  $m_0 > \mu > 2 \sqrt{\Vert Ric \Vert_{L^{\infty}}}$ for the component $\chi f^{\perp}$.
\end{proof}

We shall now use lemma \ref{lem:IntermediateInequality} to prove the inequality \ref{eq:MainInequality}. This is done as follows. One minimizes the functional
$$J(f)= \Vert d_2^*f \Vert^2_{H_{0,\alpha-1}},$$
subject to the constraint that $\Vert f \Vert_{H_{0,\alpha}}=1$. The minimizer $f$ satisfies an elliptic equation of the kind
$$d_2d_2^* f + g_1 d_2^* f + g_2 f =0,$$
where $g_1,g_2$ are smooth functions. A standard iteration argument in elliptic PDE theory gives an $L^{\infty}$ bound on the minimizer $f$, over compact sets. So that one can write $\Vert f \Vert_{L^{\infty}(K)} < C$, for $C>0$ some constant and $K$ some compact set containing $\cup_{i=1}^k B_{\epsilon^i_{out}}(p_i)$. Denote by $\sigma_{d_2^*}$ the symbol of the operator $d_2^*$, then 
\begin{eqnarray}\nonumber
\Vert d_2^* (\chi f) \Vert_{0,H_{\alpha-1}}^2 & = & \Vert \sigma_{d_2^*}(d \chi) f + \chi d_2^* f \Vert^2_{H_{0,\alpha-1}} \\ \nonumber
& \leq & \Vert \sigma_{d_2^*}(d \chi) f \Vert^2_{H_{0,\alpha-1}} + \Vert \chi d_2^* f \Vert^2_{H_{0,\alpha-1}} \\ \nonumber
& \leq & C_2 \Vert d \chi \Vert^2_{H_{0,\alpha-1}} + \Vert  d_2^* f \Vert^2_{H_{0,\alpha-1}}.
\end{eqnarray}
where in the first term one used the bound $\Vert f \Vert_{L^{\infty}} <C$ on $f$ and $C_2>0$ denotes some possibly larger constant in order to also bound $\vert \sigma_{d_2}\vert$. In the last term it was used that $\Vert \chi d_2^* f \Vert^2_{H_{0,\alpha-1}} \leq \Vert d_2^* f \Vert^2_{H_{0,\alpha-1}}$. Now notice that $\vert d \chi \vert$ is supported on the union of the annuli $B_{\epsilon^i_{out}}(p_i) \backslash B_{\epsilon^i_{in}}(p_i)$, where $\vert d \chi \vert \leq Cm_0^{1/2}$. However, in this region the spaces $H_{0,\alpha-1}$ use the rescaled norm $\vert (d \chi)_{m_0^{-1}} \vert_{g_{m_0^{-1}}} \leq C m_0^{-3/2}$, see equation \ref{FunctionSpaces1} and the definitions preceding it. Moreover, having into account the volume of the annuli one obtains $\Vert d \chi \Vert^2_{H_{0,\alpha-1}} \leq Cm_0^{-3}$, so that
$$\Vert  d_2^* f \Vert^2_{H_{0,\alpha-1}} \geq \Vert d_2^* (\chi f) \Vert_{H_{0,\alpha-1}}^2 - C_3 m_0^{-3},$$
for some $C_3>0$. Insert here the intermediate inequality from proposition \ref{lem:IntermediateInequality}
\begin{eqnarray}\nonumber
\Vert  d_2^* f \Vert^2_{H_{\alpha-1}} & \geq & c \Vert \chi f \Vert_{H_{1,\alpha}}^2 - C_3 m_0^{-3} \\ \nonumber
& \geq & c\Vert f \Vert_{H_{0,\alpha}}^2- c \sum_{i=1}^{k} \Vert f \Vert_{H_{0,\alpha}(B_{\epsilon^i_{out}(p_i)})}^2 - C_3 m_0^{-3}.
\end{eqnarray}
The first term is $c$ as $f$ satisfies the constraint that $\Vert f \Vert_{H_{0,\alpha}}=1$. The last two ones can be made arbitrarily small using the $L^{\infty}$ bound on $f$ over compact sets and by letting $m_0$ get big. Hence one obtains the inequality in equation \ref{eq:MainInequality}, then the discussion above it proves the main proposition \ref{prop:LinearEquation} of this section.

\subsection{The Nonlinear Terms}\label{NonlinearEquation}

Equipped with the surjectivity of the operator $d_2$, proved in proposition \ref{prop:LinearEquation}, we will now apply the Implicit Function Theorem to solve the monopole equation $\ast F_{A_{0}+a} - \nabla_{A_0 + a} \left( \Phi_0 + \phi \right)=0$. This can be written in the form
$$e_0 + d_2 (a,\phi) + N((a,\phi),(a,\phi))=0,$$
where $e_0$ is the error term of the approximate solution $(A_0,\Phi_0)$, $d_2$ the linearisation of the monopole equation at $(A_0,\Phi_0)$ and $N$ is the nonlinear term $N((a, \phi), (b, \psi))= \ast \frac{1}{2}\left[ a \wedge b \right] - \frac{1}{2}\left[ a ,\psi \right]- \frac{1}{2}[b,\phi]$. We shall look for a solutions $(a,\phi)$ in the image of the right inverse to $d_2$, i.e. $(a,\phi)= Qu$ for some $u$ a $1$-form with values in $\mathfrak{g}_P$, such that $u \in H_{0,\beta-1}$, with $\beta > -1$. Then, the equation is
\begin{equation}\label{eq:SchematicMonopole}
u + N(Qu,Qu)= - e_0,
\end{equation}
which is solved using the following particular version of the contraction mapping theorem

\begin{lemma}\label{lemma:Nonlinear}
Let $B$ be a Banach space and $q : B \rightarrow B$ a smooth map such that for all $u,v \in B$
$$\Vert q(u)-q(v) \Vert \leq k \left( \Vert u \Vert + \Vert v \Vert \right)  \Vert u - v \Vert,$$
for some fixed constant $k$ (i.e. independent of $u$ and $v$). Then, if $ \Vert v \Vert \leq \frac{1}{10k}$ there is a unique solution $u$ to the equation
\begin{equation}\label{quadraticeq}
u + q(u) =v,
\end{equation}
which satisfies the bound $\Vert u \Vert \leq 2 \Vert v \Vert$. 
\end{lemma}

In this lemma one interprets \ref{quadraticeq} as a fixed point equation. Under those hypothesis, the contraction mapping principle applies and one obtains the estimate by keeping track of the norms in the iterations, see \cite{Jaffe1980}. This will be used to prove

\begin{proposition}\label{prop:MonopoleSolution}
Let $k \in \mathbb{N}$ and $(X,g)$ be asymptotically conical with $b^2(X)=0$. Then, there is $\mu >0$, such that for all $m \in H^0(\Sigma, \mathbb{R})$ with $m \geq \mu$ and $(p, \alpha , \theta) \in X^k(m) \times H^1(X, \mathbb{S}^1) \times \check{\mathbb{T}}^{k-1}$. There is a unique charge $k$ and mass $m$ monopole $(A, \Phi)$, which can be written as $(A,\Phi)= (A_0,\Phi_0) + Qu$ and satisfies
$$\Vert (A,\Phi) - (A_0,\Phi_0) \Vert_{H_{1,-\frac{1}{2}}} \leq C m^{-\frac{7}{4}},$$
where $(A_0,\Phi_0)=H(p, \alpha, \theta)$ is the approximate solution from proposition \ref{prop:AproximateSolution} and $C$ denotes a constant independent of $m$.
\end{proposition}

\begin{lemma}\label{lemma:HSInequalities}
Let $\beta = - 1/2 \not\in D(d+d^*)$. Then, there is $C >0$, such that for all $f,g \in H_{1, \beta}$
$$\Vert N(f,g) \Vert_{H_{0, \beta-1}} \leq C \Vert f \Vert_{H_{1,\beta}} \Vert g \Vert_{H_{1,\beta}}.$$
\end{lemma}
\begin{proof}
That $\beta = - \frac{1}{2} \not\in D(d+d^*)$ follows from the first item in lemma \ref{DiracBoundII} and the fact that the Laplacian $\Delta_{\Sigma}$ has positive eigenvalues. To prove the inequality, let $\chi$ be the bump function \ref{eq:BumpFunction} from last section and write $N(f,g)$ as the sum of two components, one supported inside the balls $B_{\epsilon}(p_i)$ and the other one outside these. We shall prove the inequality separately for each of these two components.
\begin{enumerate}
\item If $f,g$ are supported inside the ball $B_{\epsilon^i_{out}}(p_i)$, then using both the Hold\"er and Sobolev inequalities
\begin{eqnarray}\nonumber
\Vert N(f,g) \Vert_{H_{0, \beta-1}} & = & \Vert m_0^{-2} fg \Vert_{L^2(B_{\epsilon}(p_i))} \leq  C \Vert m_0^{-1} f \Vert_{L^4(B_{\epsilon}(p_i))} \Vert m_0^{-1} g \Vert_{L^4(B_{\epsilon}(p_i))} \\ \nonumber
& \leq & C' \Vert  m_0^{-1} f \Vert_{L^2_1(B_{\epsilon}(p_i))} \Vert  m_0^{-1} g \Vert_{L^2_1(B_{\epsilon}(p_i))} \\ \nonumber
& \leq & C''  \Vert f \Vert_{H_{1, \beta}} \Vert  g \Vert_{H_{1, \beta}}.
\end{eqnarray}
Moreover the constant $C''>0$ is independent of $m_0$.

\item For $f,g$ supported on the big open set $U= X \backslash \cup_{i=1}^k \overline{B_{\epsilon}(p_i)}$, we can write $f= f^{\Vert} + f^{\perp}$, and similarly for $g$. Then
$$N(f,g)= \left( N(f^{\Vert}, g^{\perp}) + N(f^{\perp}, g^{\Vert}) \right) + N(f^{\perp}, g^{\perp}),$$
where the first two terms have values in $\mathfrak{g}_P^{\perp}$, while the third has values in $\mathfrak{g}_{P}^{\Vert}$. We shall now bound these separately. For $\beta \leq - \frac{1}{2}$ the Hold\"er and Sobolev inequalities give
\begin{eqnarray}\nonumber
\Vert N(f^{\perp},g^{\Vert}) \Vert_{H_{0, \beta-1}} & \leq & C \Vert  \phi_D N(f^{\perp},g^{\Vert}) \Vert_{L^2} \\ \nonumber
& \leq & C \Vert \rho^{\beta+ \frac{3}{2} - \beta -\frac{3}{3} - \frac{3}{6}}\phi_D N( f^{\perp},g^{\Vert}) \Vert_{L^2} \\ \nonumber
& \leq & C' \Vert \rho^{\beta+ \frac{3}{2}-1} \phi_D f^{\perp}\Vert_{L^{3}} \Vert \rho^{- \beta - \frac{3}{6}}g^{\Vert} \Vert_{L^{6}} \\ \nonumber
& \leq &  C' \Vert  \phi_D f^{\perp}\Vert_{L^{3}} \Vert g^{\Vert} \Vert_{L^{6}_{0, \beta}} \\ \nonumber
& \leq &  C'' \Vert  \phi_D f^{\perp}\Vert_{L^2_1} \Vert g^{\Vert} \Vert_{L^{2}_{1, \beta}} ,
\end{eqnarray}
where it was used that $\beta+ \frac{3}{2}-1= \beta + \frac{1}{2} \leq 0$, for $\beta \leq -\frac{1}{2}$ and that $\Vert g^{\Vert} \Vert_{L^{6}_{0, \beta}} = \Vert \rho^{- \beta - \frac{3}{6}}g^{\Vert} \Vert_{L^{6}}$, by definition. Since the weighted $H$-norm uses standard $L^2_1$ spaces along the components in $\mathfrak{g}_P^{\perp}$ and weighted norms along the $\mathfrak{g}_P^{\Vert}$ one can bound the last item above as
\begin{eqnarray}\nonumber
\Vert N(f^{\perp},g^{\Vert}) \Vert_{H_{0, \beta-1}} & \leq &  C'' \Vert  f^{\perp}\Vert_{H_{1, \beta}} \Vert g^{\Vert} \Vert_{H_{1, \beta}} .
\end{eqnarray}
The term $N(f^{\Vert}, g^{\perp})$ follows by a similar computation. To evaluate the other term, which lies in $\mathfrak{g}^{\Vert}$ note that for $\beta \geq -\frac{1}{2}$
\begin{eqnarray}\nonumber
\Vert N(f^{\perp},g^{\perp}) \Vert_{H_{0, \beta-1}} & \leq & C \Vert \rho^{-\beta+1-\frac{3}{2}} N(f^{\perp},g^{\perp}) \Vert_{L^2} \\ \nonumber
& \leq & C \Vert N( f^{\perp},g^{\perp}) \Vert_{L^2},
\end{eqnarray}
hence one can once again use the H\"older and Sobolev inequalities to bound the above quantity by $\frac{1}{m_0^2}\Vert \phi_D f^{\perp} \Vert_{L^2_1} \Vert \phi_D g^{\perp} \Vert_{L^2_1} \leq C \Vert f \Vert_{H_{1, \beta}} \Vert g \Vert_{H_{1, \beta}}$, for $m_0 > \sqrt{C}$.
\end{enumerate}
\end{proof}

\begin{proof}(of proposition \ref{prop:MonopoleSolution})
We start by showing that $q(u)=N(Qu,Qu)$ satisfies the hypothesis stated in lemma \ref{lemma:Nonlinear}, for $B=H_{0,\beta-1}$ and $\beta=-1/2$. This is done by computing
\begin{eqnarray}\nonumber
\Vert N(Qu,Qu) - N(Qv,Qv) \Vert_{H_{0,-\frac{3}{2}}} & = & \Vert  N(Q(u+v),Q(u-v)) \Vert_{H_{0,-\frac{3}{2}}} \\ \nonumber
& \leq & C   \Vert  Q(u+v) \Vert_{H_{1,-\frac{1}{2}}} \Vert  Q(u-v) \Vert_{H_{1,-\frac{1}{2}}}\\ \nonumber
& \leq &  C \left( \Vert  u \Vert_{H_{0,-\frac{3}{2}}} + \Vert  v \Vert_{H_{0,-\frac{3}{2}}} \right) \Vert  u-v \Vert_{H_{0,-\frac{3}{2}}},
\end{eqnarray}
where the constant $C$ may possibly change from one line to the other. The only hypothesis of lemma \ref{lemma:Nonlinear} which remains to be checked is that it is possible to make $\Vert e_0 \Vert_{H_{0,-\frac{3}{2}}} \leq \frac{1}{10 C}$. This is an immediate consequence of corollary \ref{cor:ErrorEstimate}. This states the existence of an approximate monopole $H(p, \alpha , \theta)$, whose error term $e_0$ satisfies $W_0 \vert e_0 \vert_{g} \leq c m_0^{-1}$, for some constant $c>0$ not depending on $m_0$. Moreover, $\vert e_0 \vert$ is supported on $\cup_{i=1}^k B_{\epsilon^i_{out}}(p_i)$
\begin{eqnarray}\nonumber
\Vert e_0 \Vert_{H_{0,-\frac{3}{2}}}^2 & = & \int_{X} \vert W_0 e_0 \vert^2 \leq c k m_0^{-2} Vol(B_{\epsilon^i_{out}(p_i)}) \leq ck m_0^{-7/2}
\end{eqnarray}
and so for big $m_0$ one obtains a solution $u$ to equation \ref{eq:SchematicMonopole}. In fact from lemma \ref{lemma:Nonlinear}, there is a constant $c'$ such that this is the unique solution $u$ to equation \ref{eq:SchematicMonopole} with $\Vert  u \Vert_{H_{0,-\frac{3}{2}}} \leq c' m_0^{-7/4}$. Then by setting $(A,\Phi)=(A_0+a,\Phi_0+\phi)$, where $(a,\phi)= Qu$, one obtains a monopole which is the unique one of this form satisfying
$$\Vert (A,\Phi) - (A_0,\Phi_0) \Vert_{H_{1,-\frac{1}{2}}} \leq C m_0^{-7/4}.$$
\end{proof}

%\begin{problem}
%Find a right inverse to the operator $d_2$ modulo harmonic one forms. Solution: Let $n$ be the dimension of the space of harmonic one forms with a certain decay and $\omega_{i}$ basis elements. Then, get a right inverse for
%$$\tilde{d_2}((a, \phi), (x_1,...,x_n))= d_2(a,\phi) + \sum x_i \omega_i.$$
%\end{problem}
%\begin{problem}
%Using the solution to the problem above to find monopoles when $H^2 \neq 0$. This amounts to solve the equation
%$$e_0 + d_2 (a,\phi) + N((a,\phi),(a,\phi))=\sum x_i \omega_i.$$
%Then one solves an infinite dimensional equation as above together with an equation in $\mathbb{R}^n $. In fact $\mathbb{R}^n \cong H^1_{cs}(X, \mathbb{R}) \cong H^2(X, \mathbb{R})$ by proposition \ref{prop:Marshal}. 
%\end{problem}

\section{Uniqueness in a Gauge Orbit}

Proposition \ref{prop:MonopoleSolution} can be used to make the following

\begin{definition}
Let $k\in \mathbb{N}$ and $(X,g)$ be asymptotically conical with $b^2(X)=0$ and $m \in H^0(\Sigma, \mathbb{R}) \cong \mathbb{R}$, be such that $m \geq \mu$, so that proposition \ref{prop:MonopoleSolution} applies and defines a map
\begin{equation}
\tilde{h}: X^k(m) \times  H^1(X, \mathbb{S}^1) \times \check{\mathbb{T}}^{k-1} \rightarrow M_{k,m}.
\end{equation}
\end{definition}

The image of the map $\tilde{h}$ above corresponds to monopoles $(A,\Phi)$, which can be written as $(A_0,\Phi_0) + Qu$, where $\Vert Qu \Vert_{H_{1, -1/2} } \leq Cm^{-7/4}$ and $(A_0,\Phi_0)$ is an approximate solution in the image of the map $H$ from proposition \ref{prop:AproximateSolution}. The goal of this section is to show that given $(A, \Phi)$ in the image of $\tilde{h}$, all gauge equivalence classes of monopoles close to $(A, \Phi)$ come from this construction, i.e. to show

\begin{theorem}\label{th:MainTh2}
The map $\tilde{h}$ descends to a local diffeormorphism
\begin{equation}
h: X^k(m) \times H^1(X, \mathbb{S}^1) \times \check{\mathbb{T}}^{k-1} \rightarrow \mathcal{M}_{k,m}.
\end{equation}
\end{theorem}

The rest of this section is dedicated to prove this theorem. We shall now introduce some notation, let $P=  X^k(m) \times  H^1(X, \mathbb{S}^1) \times \check{\mathbb{T}}^{k-1}$ denote the parameter space, $x \in P$ be a point and $(A, \Phi)= \tilde{h}(x)$, then $[(A,\Phi)] =h(x) \in \mathcal{M}_{k,m}$ is the corresponding gauge equivalence class. Moreover, $(A, \Phi)=(A_0,\Phi_0) + Qu$, where $(A_0, \Phi_0)=H(x)$ and $\Vert Qu \Vert_{H_{1, -1/2} } \leq Cm^{-7./4}$.\\
Invoking the implicit function theorem, we may take a model for $\mathcal{M}_{k,m}$ as a slice in $C_{k,m}$. This identifies a neighborhood of $h(x)$ in $\mathcal{M}_{k,m}$ with a neighborhood of zero in $\ker (d_1^* \oplus d_2)_{-1/2} \subset T_{\tilde{h}(x)}C_{k,m}$. Since $d_1^* \oplus d_2$ is Fredholm \cite{Kottke13}, we can fix an orthonormal basis of $\ker (d_1^* \oplus d_2)_{-1/2}$ and define a projection
$$\pi_h : T_{\tilde{h}(x)}C_{k,m} \rightarrow \ker (d_1^* \oplus d_2)_{-\frac{1}{2}}.$$
Then it is enough to show that there are neighborhoods $U$ of $x \in P$ and $V$ of $0 \in \ker (d_1^* \oplus d_2)_{-1/2}$, such that
$$\pi_h \circ \tilde{h}: U \rightarrow V$$
is a diffeomorphism.

\begin{lemma}\label{lem:FirstLemmaLD}
There is $\mu >0$ and $R >0$ such that for all $m \geq \mu$ the map
$$\pi_h \circ H: P \rightarrow \ker (d_1^* \oplus d_2)_{-1/2},$$
maps an open neighborhood $U'$ of $x \in P$ onto $B_R$ a ball of radius $R$ in $\ker (d_1^* \oplus d_2)_{-1/2} \subset C_{k,m}$, equipped with the metric $H_{1,-1/2}$, i.e.
$$B_R = \lbrace u \in \ker (d_1^* \oplus d_2)_{-1/2} \ \vert \ \Vert u \Vert_{H_{1, -1/2}} < R \rbrace.$$
\end{lemma}
\begin{proof}
Notice that given $x \in P$, the map $dH_x : TP \rightarrow T_{H(x)}C_{k,m}$ is injective and we can define the norm
$$\Vert v \Vert_{P} = \Vert dH_x(v) \Vert_{H_{1,-\frac{1}{2}}},$$
where $v \in T_xP$ and $dH_x(v)= \frac{d}{dt}\Big\vert_{t=0} H( \exp_{x}(tv) )$. It follows from \cite{Kottke13} that $d_1^* \oplus d_2$ has index $i=4k + b^1(X)-1$ (our weight $\beta=-\frac{1}{2}$ corresponds in that reference to the weight $\gamma= -\beta -\frac{3}{2}=-1$). Moreover, the same proof as that of proposition \ref{prop:LinearEquation} gives that $d_1^* \oplus d_2$ is surjective and so we have $\dim(\ker (d_1^* \oplus d_2))= \dim (T_xP)$. Hence, in order to show that $\pi_h \circ H$ is an isomorphism it is enough to show that it is injective. In fact we shall prove an inequality of the form
\begin{equation}\label{eq:LocalDiffIneq}
\Vert \pi_h \circ dH_x(v)  - dH_x(v) \Vert_{H_{1,\beta}} \leq c(m) \Vert v \Vert_{P},
\end{equation}
with $c(m)>0$ converging to zero as $m \rightarrow +\infty$ and all $v \in T_xP$. To prove this inequality notice that the operator $d_1^* \oplus d_2$ has closed image, $\pi_h (dH_x (v))$ is in its kernel and $ \pi_h dH_x (v) - dH_x (v)$ is orthogonal to it. So, there is a constant $c>0$ such that for all $v$
$$\Vert d_1^* \oplus d_2 \left( \pi_h dH_x (v) - dH_x (v)\right) \Vert_{H_{0,\beta-1}} \geq c \Vert  \pi_h dH_x (v) - dH_x (v)  \Vert_{H_{1,\beta}},$$
and the left hand side is equal to $\Vert d_1^* \oplus d_2  ( dH_x(v)) \Vert_{H_{0,\beta-1}}$ which by construction of the map $H$ is no greater than $c(m) \Vert dH_x (v) \Vert_{H_{1,\beta}}= c(m) \Vert v \Vert_P$, with $c(m) \rightarrow 0$ as $m \rightarrow + \infty$. Hence, for large enough $m$, the map $\pi_h \circ dH$ is a degree $1$ map of a sphere of radius $R_1$ in $T_xP$ into a topological sphere in the annulus of in $\ker( d_1^* \oplus d_2)$ whose inner and outer radius are $R_1-c(m)$ and $R_1+c(m)$ respectively. Hence, by the implicit function theorem, for large $m$, $\pi_h \circ H$ is a local diffeomorphism and maps an $m$ independent neighborhood $U'$ of $x$ onto a neighborhood of $H(x)$. Moreover, one can use equation \ref{eq:LocalDiffIneq} to prove that there is an inequality
$$\Vert \pi_h dH_x (v)  \Vert_{H_{1,\beta}} \geq (1- c(m) ) \Vert v \Vert_{P}$$
and so the derivative $\pi_h \circ dH_x$ is uniformly bounded by below for large $m$. This implies that for large $m$, $\pi_h \circ H(U')$ is a neighborhood of $\pi_h \circ H(x)$ with a fixed size, in particular it contains a ball of a fixed radius $B_R$ around $\pi_h (H(x))$.
\end{proof}

%In the same way we did in the discussion preceding the statement of the lemma, writing $\tilde{h}(x)=H(x) + Qu$ we shall let $\ker (d_1^* \oplus d_2)^H_{-\frac{1}{2}}$ denote the kernel of the operator $(d_1^* \oplus d_2)^H$ computed using the configuration $H(x)$ and
%$$\pi_H : T_{\tilde{h}(x)}C_{k,m} \rightarrow \ker (d_1^* \oplus d_2)^H_{-\frac{1}{2}},$$
%be the orthogonal projection onto it. Notice that $d_1^* \oplus d_2 = d_1^* \oplus d_2 + Qu *$, where $Qu *$ denotes the algebraic operator using $Qu$. Due to elliptic regularity we can obtain $L^2$-bounds in all derivatives of $Qu$, hence $Qu *$ defines a compact operator whose norm converges to zero as $m \rightarrow +\infty$. It follows from this that
%$$\Vert \pi_{h} - \pi_H \Vert \leq c'(m),$$
%with $c'(m) \rightarrow 0$ as $m \rightarrow + \infty$. 

Recalling the discussion preceding the statement of the previous lemma, we shall now go on to prove that there are $U$ and $V$ such that
$$\pi_h \circ \tilde{h}: U \rightarrow V$$
is a diffeomorphism. We write $\tilde{h}(x)=H(x) + Qu$ with $Qu$ given by proposition \ref{prop:MonopoleSolution}. Hence, it follows from elliptic regularity that we can obtain $H_{k,-\frac{1}{2}}$ bounds in on $Qu$ for all $k$, Moreover, these are such that converge to $0$ as $m \rightarrow \infty$, in particular
$$\Vert Qu \Vert_{H_{1, -\frac{1}{2}}} \leq Cm^{-\frac{7}{4}}.$$
Hence $\pi_h (\tilde{h}(x))= \pi_h (H(x))+ \pi_h (Qu)$ and $\Vert \pi_h (Qu) \Vert_{H_{1, -\frac{1}{2}}}$ is very small. It follows that the equation $\pi_h (\tilde{h}(x))=y$, for $y$ contained in a ball of radius $R'<R$ always has a unique solution provided that $Cm^{-\frac{7}{4}} < R-R'$. This can always be achieved by taking $m$ to be large enough, namely
$$m> \mu=\left(\frac{C}{R-R'} \right)^{\frac{7}{4}}.$$
Hence we can take $V=B_{R'}$ and $U \subset U'$ to be its inverse image via $\pi_h \circ \tilde{h}$. This finishes the proof of theorem \ref{th:MainTh2}.

%===============================================================================
%\bibliography{refs}
%===============================================================================

\end{document}